[1994/12/01]
\documentclass[manuscript, reqno]{aomart}

\chardef\bslash=`\\ 

\newtheorem[{}\it]{thm}{Theorem}[section]

\newtheorem{lem}[thm]{Lemma}
\newtheorem{prop}[thm]{Proposition}

\theoremstyle{remark} 
\newtheorem{remark}[thm]{Remark}

\usepackage{dsfont}
\usepackage{amssymb}

\usepackage{hyperref}
\usepackage[numeric]{amsrefs}

\usepackage[english]{babel}

\theoremstyle{definition}
\newtheorem{defn}{\textsc{Definition}}[section]
\newtheorem{rem}{Remark}[section]
\newtheorem*[{}\it]{notation}{Notation}

\newtheorem*[{}\it]{rest}{\textsc{Theorem}}
\newtheorem*[{}\it]{proofoflemma}{Proof of Lemma}

\usepackage[textwidth=6in,textheight=9.5in, paperwidth=7.5in,paperheight=11.25in, inner=2.5cm, outer=2cm]{geometry} 

\usepackage[usenames,dvipsnames]{xcolor}

\usepackage{multicol}

\title[]{Strict Physicality of Global Weak Solutions of a Navier-Stokes Q-tensor System with Singular Potential}
\author[]{Mark Wilkinson\footnote{(\Letter\hspace{-0.5mm}) \tiny Correspondence to: \href{mailto:wilkinson@maths.ox.ac.uk}{\nolinkurl{wilkinson@maths.ox.ac.uk}}. Address: {\em Oxford Centre for Nonlinear Partial Differential Equations, Mathematical Institute, Oxford University, 24--29 St Giles', Oxford, OX1 3LB.}}}
\givenname{Mark Wilkinson}

\givenname{Mark}
\surname{Wilkinson}

\startpage{1}
\endpage{}

\allowdisplaybreaks

\newcommand{\id}{I^{d}}
\newcommand{\iid}{\int_{I^{d}}}

\newcommand{\symn}{\mathrm{Sym}_{0}(d)}

\newcommand{\tr}{\mathrm{tr}}

\hypersetup{citecolor=red}

\usepackage{marvosym}

\begin{document}

\begin{abstract} 
We study the existence, regularity and so-called `strict physicality' of weak solutions of a coupled Navier-Stokes Q-tensor system which is proposed as a model for the incompressible flow of nematic liquid crystal materials. An important contribution to the dynamics comes from a singular potential introduced by \textsc{Ball and Majumdar} \cite{bm1} which replaces the commonly employed Landau-de Gennes bulk potential. This is built into our model to ensure that a natural physical constraint on the eigenvalues of the Q-tensor order parameter is respected by the dynamics of the system. 
\end{abstract}

\maketitle

In this article, we construct global-in-time weak solutions to the following coupled Navier-Stokes Q-tensor system on the $d$-dimensional torus, which is an adaptation of a model of \textsc{Beris and Edwards} \cite{berised1}:
\begin{displaymath}
\text{(S)}\left\{
\begin{array}{c}
 \displaystyle \frac{\partial Q}{\partial t}(x, t)+\left(u(x, t)\cdot \nabla\right)Q(x, t)-S(Q(x, t), \nabla u(x, t)) \vspace{2mm} \\ \displaystyle = \Gamma\left(L\Delta Q(x, t)-\theta\frac{\partial\psi}{\partial Q}(Q(x, t))+\frac{\theta}{d}\mathrm{tr}\left[\frac{\partial\psi}{\partial Q}(Q(x, t))\right]I + \kappa\,Q(x, t)\right), \\  \\ \displaystyle \frac{\partial u}{\partial t}(x, t)+\left(u(x, t)\cdot\nabla\right)u(x, t) = \nu\Delta u(x, t) - \nabla p(x, t)+\mathrm{div}\,\left(\tau(x, t)+\sigma(x, t)\right) \\ \\  \nabla\cdot u(x, t) = 0.
\end{array}
\right.
\end{displaymath}
The dimension $d$ is either 2 or 3, $Q$ is a $d\times d$ matrix-valued map and $\Gamma, L, \theta, \kappa, \nu >0$ are constants. Moreover, $\psi$ is a convex map defined on symmetric and traceless $d\times d$ matrices, whose construction was recently given by \textsc{Ball and Majumdar} \cite{bm1} and which ensures the eigenvalues of the tensor field remain in a so-called \emph{physical interval}. The `tumbling' term $S$ and the stress tensors $\tau$ and $\sigma$ are defined in Section \ref{models}. Employing the map $\psi$ in system (S) allows us to infer that weak solutions $Q(\cdot, t)$ belong to $L^{\infty}$ for $t>0$. By virtue of the convexity of $\psi$, a suitable comparison principle argument allows us to infer \emph{strict physicality} of weak solutions (see Section \ref{secfour}, Stage III below), which in turn allows us to prove higher regularity of global weak solutions in dimension 2.

We now discuss the modelling problem in condensed matter physics which motives our study of the above coupled system. The analysis of system (S) begins in Section \ref{models}, so the mathematically-minded reader may wish to skip Section \ref{introduction} below.

\section{Motivation: Order Parameters and Physical Constraints}\label{introduction}

Nematic liquid crystals form a class of condensed matter systems whose constituent rod-like molecules give rise to rich nonlinear phenomena, such as isotropic-nematic phase transitions. Thermotropic nematic liquid crystals form a subclass whose optical properties change dramatically with variation of system temperature. Above a certain temperature threshold the constituent molecules are randomly oriented (the \emph{isotropic} phase), whereas below this threshold they tend to lie in locally preferred directions (the \emph{nematic} phase). 

It is a formidable problem to formulate a mathematically rigorous model of such phenomena based on classical molecular dynamics in the continuum. This is, naturally, due to the high dimensionality of the phase space in which the dynamics take place. Due to the forbidding complexity of such models, one is led to consider more tractable mesoscopic or macroscopic models which are built by employing the general principles of thermodynamics and classical mechanics.   
\subsection{Order Parameters}
One of the first steps to take when formulating such effective static or dynamic continuum theories is to select an \emph{order parameter} that captures the essential small scale structure in nematic systems. For meso-scale models such as the Maier-Saupe theory of statics \cite{chandr1} or dynamic Doi theories \cite{doi}, one typically posits as an order parameter a probability density $\rho$ on the unit sphere $\mathbb{S}^{d-1}$, which is to be regarded as encoding the average molecular orientation at material points in the spatial domain. As regards macro-scale models, there are a number of competing order parameters in the literature. Within the well-established Oseen-Frank \cite{frank1} or Ericksen-Leslie theories \cite{leslie1}, for example, one studies the respective static and dynamic properties of vector fields $n:\mathbb{R}^{d}\rightarrow\mathbb{R}^{d}$ taking values in the unit sphere, where $n(x)\in\mathbb{S}^{d-1}$ is known as a \emph{director}. 

The director-field formalism may however be viewed as restrictive, as the order parameter $n$ cannot account for {\em biaxiality} of liquid crystal configurations. The Nobel prize-winning physicist Pierre-Gilles de Gennes suggested a traceless and symmetric $d\times d$ matrix with real components to be more appropriate in this regard for the modelling of nematic materials.

\subsection{The Q-Tensor Order Parameter}

We now present some basic properties of the Q-tensor order parameter. For further details, consult \textsc{De Gennes and Prost} \cite{degennes1}, \textsc{Majumdar} \cite{maj1} or \textsc{Newton and Mottram} \cite{mottram1}. In what follows, the spatial dimension $d$ will be either 2 or 3, the two-dimensional case corresponding to, for instance, thin films of nematic materials. 

\subsection{Physical Constraints: an Eigenvalue Inequality}\label{explan}
Suppose that to each point $x$ in a material domain $\Omega\subseteq\mathbb{R}^{d}$ we associate a probability density function $f_{x}$ on molecular orientations which lie in $\mathbb{S}^{d-1}$. In order to capture the $\mathbb{Z}_{2}$ `head-to-tail' symmetry of nematic molecules, each density is endowed with the antipodal symmetry $f_{x}(\omega)=f_{x}(-\omega)$ for all $\omega\in\mathbb{S}^{d-1}$. The \emph{Q-tensor order parameter} $Q(x)$ at $x\in \Omega$ is defined to be
\begin{equation*}
Q(x):=\int_{\mathbb{S}^{d-1}}\left(\omega\otimes\omega - \frac{1}{d}I\right)f_{x}(\omega)\,d\omega.
\end{equation*}
It is the normalised matrix of second-order moments of the probability measure on $\mathbb{S}^{d-1}$ with density $f_{x}$. One may quickly check that $Q(x)$ is a member of \begin{equation}\label{symnaughtdee}
\symn:=\left\{Q\in\mathbb{R}^{d\times d}\,:\,Q^{T}=Q\hspace{2mm}\text{and}\hspace{2mm}\mathrm{tr}\left[Q\right]=0\right\}.
\end{equation}
The term $-1/d I$ (which contains no information about the system) is included by convention in the definition so as to render the Q-tensor identically zero when $f_{x}$ is the uniform distribution on the unit sphere, corresponding to the \emph{isotropic} phase of a nematic material. In this way, the Q-tensor order parameter should be interpreted as a crude measure of the deviation of a nematic system from isotropy.

If one interprets a member of $\symn$ as a normalised matrix of second moments of some probability measure on nematic orientations, the eigenvalues of this order parameter are then constrained in the following manner. By the spectral decomposition theorem, every matrix $Q\in\symn$ has the representation
\begin{equation*}
Q=\lambda_{1}e_{1}\otimes e_{1}+ ... + \lambda_{d}e_{d}\otimes e_{d},
\end{equation*}
where each $e_{i}$ is a unit-norm eigenvector of $Q$ with corresponding eigenvalue $\lambda_{i}$. If $Q$ arises from some probability density $f$, we also have the identity
\begin{equation*}
\int_{\mathbb{S}^{d-1}}\left(\omega\otimes\omega - \frac{1}{d}I\right)f(\omega)\,d\omega = \sum_{j=1}^{d}\lambda_{j}e_{j}\otimes e_{j}.
\end{equation*}
Applying the action of the above matrices to a fixed eigenvector $e_{k}$ and then producing the result again with $e_{k}$, one discovers the equality
\begin{equation*}
\lambda_{k}=\int_{\mathbb{S}^{d-1}}(\omega\cdot e_{k})^{2}f(\omega)\, d\omega - \frac{1}{d},
\end{equation*}
from which one quickly deduces that $\lambda_{k}$ is constrained to satisfy the inequality $-1/d\leq \lambda_{k}\leq 1-1/d$ for $k=1, ..., d$. The cases of equality correspond to {\em perfect} crystalline nematic alignment, and so are excluded. Thus, the eigenvalues $\lambda_{1}, ..., \lambda_{d}$ of any such $Q\in\symn$ should satisfy
\begin{equation}\label{physicalregime}
-\frac{1}{d}< \lambda_{k}< 1-\frac{1}{d}\quad \text{for} \quad k=1, ..., d.
\end{equation}
We subseqently refer to all Q-tensors whose eigenvalues satisfy inequality \eqref{physicalregime} as \emph{physical}. 

From a modelling perspective, the physicality requirement \eqref{physicalregime} on the order parameter $Q$ presents a problem. It is often not clear (and sometimes not even the case) that solutions of static and dynamic theories of nematics which employ the Q-tensor order parameter respect this physicality condition. Clearly, an issue then arises as to how one should interpret the solutions of such theories in a meaningful way. 

One well-studied theory that employs the Q-tensor order parameter is the static Landau-de Gennes theory. The Landau-de Gennes energy in the case $d=3$ is given by
\begin{equation*}
E_{\mathrm{LdG}}[Q]:=\int_{\Omega}\left(\frac{L}{2}|\nabla Q|^{2}+\frac{a}{2}\mathrm{tr}\left[Q^{2}\right]-\frac{b}{3}\mathrm{tr}\left[Q^{3}\right]+\frac{c}{4}\left(\mathrm{tr}\left[Q^{2}\right]\right)^{2}\right)\,dx,
\end{equation*}
for a suitably smooth domain $\Omega\subseteq\mathbb{R}^{3}$ and with $a\in\mathbb{R}$ and $b, c>0$. In \cite{maj1}, \textsc{Majumdar} shows that bulk energy minimisers in the {\em spatially-homogeneous} Landau-de Gennes theory ($\nabla Q\equiv 0$) do not always respect the above physical constraint on eigenvalues. In particular, when laboratory values for the liquid crystal material MBBA are strapped to the material-dependent constants $(a, b, c)$, one finds that \eqref{physicalregime} is not respected 2$^{\circ}$C below the isotropic-nematic phase transition temperature.

Aside from such issues related to the practical value of the unconstrained Landau-de Gennes theory for modelling real nematic materials, one can appreciate from this simple example of \textsc{Majumdar} that it is not immediately clear in the more complicated case of {\em inhomogeneous} spatial profiles whether or not minimisers of $E_{\mathrm{LdG}}$ subject to physical boundary conditions are physical pointwise on the spatial domain $\Omega$.

Na\"{i}vely, one might hope to circumvent this problem in general simply by restricting the class of candidate minimising maps to
\begin{equation*}
\mathcal{A}:=\left\{Q\in H^{1}(\Omega)\,:\, -\frac{1}{d}<\lambda_{k}(Q(x))<1-\frac{1}{d} \quad\text{for}\hspace{2mm} k=1, ...,  d \hspace{2mm}\text{and almost every}\hspace{2mm} x\in\Omega\right\},
\end{equation*}
for instance. Even if one were to consider this new model, the problem remains for the case of dynamics. It is not an easy task in general to show that dynamics generated by physically-relevant equations preserve the convex set
\begin{equation*}
\triangle_{d}:=\left\{Q\in\symn\,:\, -\frac{1}{d}<\lambda_{i}(Q)<1-\frac{1}{d}\hspace{2mm}\text{for}\hspace{2mm} i= 1, ..., d\right\},
\end{equation*} 
i.e. that $Q(x, t)\in \triangle_{d}$ for almost every $x$ and $t>0$. The following approach of \cite{bm1}, which has its roots in the paper of \textsc{Katriel et al.} \cite{katriel}, allows one to treat the modelling issue in both statics and dynamics in the same manner.

\subsection{The Ball-Majumdar Singular Potential}
Recently, \textsc{Ball and Majumdar} \cite{bm1} proposed a qualitatively-similar continuum theory, for which the authors effectively `build in' physicality of the Q-tensor. 

We now briefly outline their construction of a singular map $\psi: \symn\rightarrow \mathbb{R}\cup\{\infty\}$ and then introduce the problem in statics which motivates our present problem in dynamics. Although the construction of $\psi$ in \cite{bm1} is performed in dimension $d=3$ alone, it may be generalised in a straightforward manner to the case $d=2$.

\subsection{The Roots of $\psi$ in Maier-Saupe Theory}
As a starting point, consider the spatially-homogeneous Maier-Saupe mean field theory, 
\begin{equation*}
I_{\mathrm{MS}}[\rho]:=\theta\int_{\mathbb{S}^{d-1}}\rho(\omega)\log{\rho(\omega)}\,d\omega + \frac{\kappa}{2}\int_{\mathbb{S}^{d-1}}\int_{\mathbb{S}^{d-1}}\left(\frac{1}{d}-(\omega\cdot\alpha)^{2}\right)\rho(\omega)\rho(\alpha)\,d\omega d\alpha,
\end{equation*}
for suitable probability density functions $\rho:\mathbb{S}^{d-1}\rightarrow [0, 1]$ satisfying the symmetry property $\rho(\omega)=\rho(-\omega)$. Moreover, $\theta$ represents temperature and the constant $\kappa$ encodes to some extent information on molecular interactions. A routine calculation reveals that $I_{\textrm{MS}}$ has the form
\begin{equation*}
I_{\mathrm{MS}}[\rho]:=\theta\int_{\mathbb{S}^{d-1}}\rho(\omega)\log{\rho(\omega)}\,d\omega -\frac{\kappa}{2}\tr\left[Q^{2}\right],
\end{equation*}
where $Q$ is the Q-tensor corresponding to the probability density $\rho$.

For a given \emph{physical} Q-tensor $Q\in\symn$, the authors consider the following natural entropy minimisation problem associated with $I_{\mathrm{MS}}$, namely
\begin{equation}\label{maps}
\min_{\mathcal{A}_{Q}}\int_{\mathbb{S}^{d-1}}\rho(\omega)\log{\rho(\omega)}\,d\omega,
\end{equation}
where
\begin{equation*}
\mathcal{A}_{Q}:=\left\{\rho:\mathbb{S}^{d-1}\rightarrow [0, 1]\,:\,\int_{\mathbb{S}^{d-1}}\rho(\omega)\,d\omega=1\quad\text{and}\quad \int_{\mathbb{S}^{d-1}}\left(\omega\otimes\omega-\frac{1}{d}I\right)\rho(\omega)\,d\omega=Q\right\}.
\end{equation*}
In \cite{bm1}, it is shown that this problem possesses a unique minimising density $\rho^{\ast}$ in the class $\mathcal{A}_{Q}$, given explicitly by
\begin{equation*}
\rho^{\ast}(\omega):=\frac{\exp{\left(\sum_{i=1}^{d}\mu_{i}\omega_{i}^{2}\right)}}{Z(\mu_{1}, ..., \mu_{d})}, 
\end{equation*} 
where $\mu_{1}, ..., \mu_{d}$ are Lagrange multipliers associated with the constraint that $Q$ be the Q-tensor of the density, and $Z(\mu_{1}, ..., \mu_{d})$ is a normalisation factor which ensures $\rho^{\ast}$ is of unit mass. By uniqueness of minimisers of \eqref{maps}, one can then construct a related singular map $\psi:\symn\rightarrow \mathbb{R}\cup\{\infty\}$,
\begin{displaymath}
\psi(Q):=\left\{
\begin{array}{ll}
\displaystyle\min_{\rho\in\mathcal{A}_{Q}}\int_{\mathbb{S}^{d-1}}\rho(\omega)\log{\rho(\omega)}\,d\omega & \quad \text{if} \quad \displaystyle -\frac{1}{d} < \lambda_{j}(Q)<1-\frac{1}{d}, \\ & \\ \infty & \quad \text{otherwise}.
\end{array}
\right.
\end{displaymath} 
Such a map establishes a framework in which non-physical Q-tensors are essentially forbidden. We henceforth denote the effective domain of $\psi$ by $\mathsf{D}(\psi):=\{A\in\symn\, :\, \psi(A)<\infty\}$. The following proposition records important properties of the potential $\psi$. 
\begin{prop} The map $\psi:\symn\rightarrow\mathbb{R}\cup\{\infty\}$ has the following properties:
\begin{itemize} 
\item[] \framebox{\emph{Analytic Properties}} \vspace{1mm}
\item[\textbf{(P1)}] The map $\psi$ is smooth on its effective domain, i.e. $\psi\in C^{\infty}(\mathsf{D}(\psi))$; 
\item[\textbf{(P2)}] It is bounded from below, i.e. there exists $\psi_{0}>0$ such that $-\psi_{0}\leq \psi(X)$ for all $X\in\symn$; 
\item[\textbf{(P3)}] It exhibits logarithmic blow-up as $X\rightarrow\partial\mathsf{D}(\psi)$ from the interior. \vspace{1mm}
\item[] \framebox{\emph{Geometric Property}} \vspace{1mm}
\item[\textbf{(P4)}] The map $\psi$ is convex on $\mathsf{D}(\psi)$, which is itself a convex subset of $\symn$.
\item[] \framebox{\emph{Algebraic Property}} \vspace{1mm}
\item[\textbf{(P5)}] The map $\psi$ is an isotropic function of $d\times d$ matrices, i.e. $\psi(RXR^{T})=\psi(X)$ for all $R\in\mathrm{SO}(d)$, whenever $X\in\symn$ is fixed.
\end{itemize}
\end{prop}
\begin{proof}
We refer the reader to \textsc{Ball and Majumdar} \cite{bm2} for details.
\end{proof}
Having constructed the map $\psi$, the authors replace the study of the Landau-de Gennes energy functional $E_{\mathrm{LdG}}$ with the following functional
\begin{equation*}
E_{\mathrm{BM}}[Q]:=\int_{\Omega}\left(\frac{L}{2}|\nabla Q|^{2}+\theta\,\psi(Q)-\frac{\kappa}{2}\tr\left[Q^{2}\right]\right)\,dx,
\end{equation*}
by replacing the Landau-de Gennes bulk energy density $a(\theta)\mathrm{tr}[Q^{2}]/2-b\,\mathrm{tr}[Q^{3}]/3+c\,\mathrm{tr}[Q^{4}]/4$ with the singular map $\theta\,\psi(Q)-\kappa\,\mathrm{tr}[Q^{2}]/2$ associated with the Maier-Saupe energy. The construction of $\psi$ in \cite{bm1} is especially pleasing as whilst it resolves the above modelling issue for the static theory of nematics, the map $\psi$ is by no means confined to the study of statics. Using techniques from convex analysis (see \textsc{Rockafellar} \cite{rockafellar1997convex} or \textsc{Ekeland and Temam} \cite{eketem1}) we study $\psi$ in the setting of dynamics. 

It proves important to make the following dichotomy between {\em physical} and {\em strictly physical} tensor fields $Q:\Omega\rightarrow\symn$.
\begin{defn}[{\bf Physicality}] 
The field $Q:\Omega\rightarrow\symn$ is said to be \textit{physical} on $\Omega$ if and only if $-1/d <\lambda_{i}(Q(x))<1-1/d$ for almost every $x\in\Omega$ and $i=1, ..., d$. 
\end{defn}
We compare this with the stronger notion of \emph{strict} physicality.
\begin{defn}[\bf Strict Physicality]
The field $Q:\Omega\rightarrow\symn$ is said to be \textit{strictly physical} on $\Omega$ if and only if there exists $\delta>0$ sufficiently small such that $\delta-1/d\leq \lambda_{i}(Q(x))\leq 1-1/d-\delta$ for almost every $x\in\Omega$ and $i=1, ..., d$.
\end{defn}
The reader can check that if the map $Q:\Omega\rightarrow \symn$ satisfies $\psi(Q)\in L^{1}$ for any given map satisfying \textbf{(P1)} to \textbf{(P5)} above, then it is physical. On the other hand, it is the case that $Q$ is strictly physical if and only if $\psi(Q)\in L^{\infty}$. An issue similar to the eigenvalue constraint in liquid crystal theory also arises in the theory of elasticity (see \cite{ball1}), where one is interested in demonstrating that the determinant of the deformation gradient tensor is bounded uniformly away from zero, i.e. $\mathrm{det}(\nabla y(x))\geq \gamma$ for some $\gamma>0$. However, this problem is more challenging as constraints are being placed upon {\em derivatives} as opposed to the undifferentiated field variable. 

We now introduce the problem of study in \cite{bm1} which motivates our present problem in dynamics from a modelling point of view.
\subsection{Motivation for the Problem in Dynamics}
Consider the minimisation problem
\begin{equation*}
\min_{\mathcal{B}}\int_{\Omega}\left(\frac{L}{2}|\nabla Q|^{2}+\theta\,\psi(Q)-\frac{\kappa}{2}\tr\left[Q^{2}\right]\right)\,dx,
\end{equation*}
where the appropriate candidate maps in $\mathcal{B}$ have {\em strictly physical} trace on the boundary $\partial\Omega$. One can show using a maximum principle approach that minimisers are strictly physical throughout the domain $\Omega$. We wish to answer a question which is similar in spirit for the case of dynamics. Supposing that the Q-tensor field evolves under system (S) above, we ask the following:\vspace{3mm}
\newline
\framebox{
\begin{minipage}[b]{0.965\linewidth}
\centering
If initial data $(Q_{0}, u_{0})$ are of finite energy, namely
\begin{equation*}
\mathcal{E}(Q_{0}, u_{0}):=\int_{\id}\left(\frac{1}{2}|u_{0}|^{2}+\frac{L}{2}|\nabla Q_{0}|^{2}+\theta\,\psi(Q_{0})-\frac{\kappa}{2}\mathrm{tr}\left[Q_{0}^{2}\right]\right)\,dx<\infty,
\end{equation*}
is it the case that weak solutions $Q(\cdot, t)$ are strictly physical for $t>0$, i.e. $\psi(Q(\cdot, t))\in L^{\infty}$?
\end{minipage}
}
\vspace{3mm}
\newline
This is one of the central questions of this paper and we answer it in the affirmative. Knowledge that $\psi(Q(\cdot, t))\in L^{\infty}$ allows us to infer $Q(\cdot, t)\in L^{\infty}$ by boundedness of the related eigenvalue maps. The `strict physicality' property of solutions is an attractive feature from the point of view of regularity theory. Using the fact that the range of weak solutions $Q$ then belongs to a compact subset of $\symn$, we are able to prove higher regularity of solutions that persists for all time $t>0$ in dimension 2. 

The unmodified version of system (S), taken from \textsc{Beris and Edwards} \cite{berised1}, has also been studied in the physics community by, among others, \textsc{Yeomans et al.} \cite{yeo1, yeo2} and also in the mathematical community by \textsc{Paicu and Zarnescu} \cite{zarnescu1, zarnescu2}.
\begin{rem}
One might consider the above boxed question to be natural, as it is characteristic of evolution equations of parabolic type that solutions with $L^{1}$ initial data are instantaneously in $L^{\infty}$: for more details on such topics, one could consult \textsc{Ladyzhenskaya, Solonnikov and Ural'ceva} \cite{ladyzhenskaya1}. Although one may not be able to use the evolution equation for the tensor field to close an equation for the quantity $\psi(Q)$, one can nevertheless derive a parabolic \emph{inequation} satisfied by $\psi(Q)$, from which a suitable comparison principle argument yields strict physicality of weak solutions $Q$. 
\end{rem} 

\section{The Class of Models}\label{models}
In what follows, the spatial dimension $d$ will be either 2 or 3 and we employ the Einstein summation convention over repeated indices $i, j, k, \ell, m$ and $n$ with range in $\{1, ..., d\}$ throughout. Moreover, $\langle A\rangle:=A-d^{-1}\mathrm{tr}[A]I$ denotes the trace-free part of any matrix $A\in\mathbb{R}^{d\times d}$. For strictly positive constants $\Lambda, \Gamma, L, \theta, \kappa$ and $\nu$, we search for maps $Q: I^{d}\times (0, \infty)\rightarrow\mathrm{Sym}_{0}(d)$ and $u: I^{d}\times (0, \infty)\rightarrow \mathbb{R}^{d}$ on the spatial domain $I^{d}:=[-\Lambda\pi/2, \Lambda\pi/2]^{d}$ which satisfy the coupled system
\begin{displaymath}
\text{(S)}\left\{
\begin{array}{c}
 \displaystyle \frac{\partial Q}{\partial t}(x, t)+\left(u(x, t)\cdot \nabla\right)Q(x, t)-S(Q(x, t), \nabla u(x, t)) \vspace{2mm} \\ \displaystyle = \Gamma\left(L\Delta Q(x, t)-\theta\left\langle\frac{\partial\psi}{\partial Q}(Q(x, t))\right\rangle + \kappa\,Q(x, t)\right), \\  \\ \displaystyle \frac{\partial u}{\partial t}(x, t)+\left(u(x, t)\cdot\nabla\right)u(x, t) = \nu\Delta u(x, t) - \nabla p(x, t)+\mathrm{div}\,\left(\tau(x, t)+\sigma(x, t)\right) \\ \\  \nabla\cdot u(x, t) = 0,
\end{array}
\right.
\end{displaymath}
in the distributional sense. Solutions are subject to the periodic boundary conditions on the domain $\id$ and evolve from given initial data $Q_{0}:I^{d}\rightarrow\mathrm{Sym}_{0}(d)$ and $u_{0}: I^{d}\rightarrow \mathbb{R}^{d}$ to which they converge in an appropriate topology (strong-in-norm or weakly) as $t$ tends to 0 from above. 

The term $S$ expresses to what extent the flow $u$ locally `twists' and `stretches' the order parameter $Q$, and is given by
\begin{equation}\label{ess}
S(Q, \nabla u):=\left(D_{0}+\xi D\right)\left(Q+\frac{1}{d}I\right)-\left(Q+\frac{1}{d}I\right)\left(D_{0}-\xi D\right)-2\xi\left(Q+\frac{1}{d}I\right)\mathrm{tr}\left[Q\nabla u\right],
\end{equation}
where $\xi\in\mathbb{R}$ is a rotational parameter whose value will be of some significance in the course of our analysis. The tensors $D_{0}$ and $D$ defined by
\begin{equation*}
D_{0}:=\frac{1}{2}\left((\nabla u)-(\nabla u)^{T}\right) \qquad \text{and} \qquad D:=\frac{1}{2}\left((\nabla u)+(\nabla u)^{T}\right)
\end{equation*}
are the anti-symmetric and symmetric parts of the velocity gradient tensor $\nabla u$, respectively. The stress tensors $\tau$ and $\sigma$ are given component-wise by
\begin{align}\label{tau}
\tau_{ij}:= & -\xi\left(Q_{ik}+\frac{1}{d}\delta_{ik}\right)H_{kj}-\xi H_{ik}\left(Q_{kj}+\frac{1}{d}\delta_{kj}\right) \notag \\ \quad & +2\xi\left(Q_{ij}+\frac{1}{d}\delta_{ij}\right)\mathrm{tr}\left[QH\right] - L\,\mathrm{tr}\left[\frac{\partial Q}{\partial x_{i}}\frac{\partial Q}{\partial x_{j}}\right]
\end{align}
and
\begin{equation}\label{sigma}
\sigma_{ij}:= Q_{ik}H_{kj}-H_{ik}Q_{kj},
\end{equation}
where $H$ is defined for notational simplicity to be 
\begin{equation*}
H:=L\Delta Q-\theta\left\langle\frac{\partial\psi}{\partial X}(Q)\right\rangle+\kappa\,Q.
\end{equation*}
Moreover, the singular map $\psi$ belongs to the non-empty class of all maps on $\symn$ satisfying properties \textbf{(P1)} to \textbf{(P5)} above.
\begin{rem}
The right-hand side of the $Q$-equation is the formal $L^{2}$ gradient of the energy functional $E_{\mathrm{BM}}$
which replaces the usual contribution from the more commonly-adopted Landau-de Gennes energy $E_{\mathrm{LdG}}$. On thermodynamic grounds, we justify such a replacement since both bulk potentials are smooth on their respective effective domains, and are \emph{qualitatively} similar to one another, in the sense that both can describe a first-order nematic-isotropic phase transition (see Section 4 of \cite{bm1} for more on this point) and possess the same material symmetry, viz. \textbf{(P5)}.  
\end{rem}
\subsection{Statement of Main Results}\label{maintheorem}
The following two theorems contain the main results of this paper.
\begin{thm}\label{theoremone}
For initial data $(Q_{0}, u_{0})\in H^{1}\times L^{2}_{\mathrm{div}}$ of finite energy
\begin{equation*}
\mathcal{E}(Q_{0}, u_{0})=\int_{\id}\left(\frac{1}{2}|u_{0}|^{2}+\frac{L}{2}|\nabla Q_{0}|^{2}+\theta\,\psi(Q_{0})-\frac{\kappa}{2}\mathrm{tr}\left[Q_{0}^{2}\right]\right)\,dx<\infty,
\end{equation*} 
there exist maps $Q\in L^{\infty}_{\mathrm{loc}}((0, \infty); H^{1})\cap L^{2}_{\mathrm{loc}}((0, \infty); H^{2})$ and $u\in L^{\infty}_{\mathrm{loc}}((0, \infty); L^{2}_{\mathrm{div}})\cap L^{2}_{\mathrm{loc}}((0, \infty); H^{1}_{\mathrm{div}})$ which satisfy the coupled system \textnormal{(S)} in the distributional sense. The map $Q$ is also strictly physical for positive time, i.e.
\begin{equation*}
\psi(Q(\cdot, t))\in L^{\infty} \quad \text{for almost every} \quad t>0.
\end{equation*}
\end{thm}\label{theoremtwo}
Thus, if one assumes $\mathcal{E}(Q_{0}, u_{0})$ is finite, one can infer the strict physicality of weak solutions $Q(\cdot, t)$ for $t>0$. If one endows the initial data with higher regularity and insists on $Q_{0}$ being {\em strictly physical}, the following holds true.
\begin{thm}[dimension $d=2$]
If $u_{0}\in H^{1}_{\mathrm{div}}$ and $Q_{0}\in H^{2}$ with $\psi(Q_{0})\in L^{\infty}$, then there exist maps $Q\in L^{\infty}(0, T; H^{2})\cap L^{2}(0, T; H^{3})$ and $u\in L^{\infty}(0, T; H^{1}_{\mathrm{div}})\cap L^{2}(0, T; H^{2}_{\mathrm{div}})$ for any $T>0$ which satisfy \textnormal{(S)} in the sense of distributions. 
\end{thm}
\label{classomodels}



\subsection{Regularisation of $\psi$}
For the purposes of building a weak solution to the system (S), it is inconvenient to work directly with the singular map $\psi$. We work instead with a regularised map $\psi_{N}$ parameterised by the mollification index $N=1, 2, 3, ...$, such that $\psi$ can be recovered in the limit $N\rightarrow\infty$.

Firstly, for $J=1, 2, 3, ...$ we define $\psi_{J}:\symn\rightarrow \mathbb{R}$ to be the Yosida-Moreau regularisation of $\psi$, namely
\begin{equation}\label{yosidamoreau}
\psi_{J}(Q):=\min_{A\in\symn}\left(J|A-Q|^{2}+\psi(A)\right).
\end{equation}
Secondly, for fixed $J$ and for any $K=1, 2, 3, ...$ we define $\psi_{J, K}$ to be the standard mollification of the map $\psi_{J}$, namely
\begin{equation}\label{mollification}
\psi_{J, K}(Q):=K^{d^{2}}\int_{\mathbb{R}^{d\times d}}\psi_{J}(K(Q-R))\Phi(R)\,dR,
\end{equation}
where $\Phi\in C^{\infty}_{c}(\mathbb{R}^{d\times d}, \mathbb{R}_{+})$ has the unit mass property $\int_{\mathbb{R}^{d}\times\mathbb{R}^{d}}\Phi(R)\,dR=1$. Finally, we define $\psi_{N}:=\psi_{N, N}$ for each $N\geq 1$. We now quote without proof a number of properties of $\psi_{N}$, taken from \textsc{Feireisl et Al.} \cite{feireisl}, which are of use to us in the construction of weak solutions.
\begin{thm}
For each $N\geq 1$, the regularisation $\psi_{N}$ of the Ball-Majumdar potential has the following properties:
\begin{itemize}
\item[\textbf{(M1)}]\label{mone} The map $\psi_{N}$ is both $C^{\infty}$ and convex on $\mathbb{R}^{d\times d}$;
\item[\textbf{(M2)}] It is bounded from below, i.e. $-\psi_{0}\leq \psi_{N}(X)$ for all $X\in \mathbb{R}^{d\times d}$ and for all $N\geq 1$, where $\psi_{0}>0$ is the same constant appearing in \emph{\textbf{(P2)}};
\item[\textbf{(M3)}] $\psi_{N}\leq \psi_{N+1}\leq\psi$ on $\mathbb{R}^{d\times d}$ for $N\geq 1$;
\item[\textbf{(M4)}] $\psi_{N}\rightarrow \psi$ in $L^{\infty}_{\mathrm{loc}}(\mathsf{D}(\psi))$ as $N\rightarrow\infty$, and $\psi_{N}$ is uniformly divergent on $\symn\setminus\mathsf{D}(\psi)$ as $N\rightarrow\infty$;\vspace{1mm}
\item[\textbf{(M5)}] $\displaystyle\frac{\partial\psi_{N}}{\partial Q}\rightarrow \frac{\partial\psi}{\partial Q}$ in $L^{\infty}_{\mathrm{loc}}(\mathsf{D}(\psi))$ as $N\rightarrow\infty$; \vspace{1mm}
\item[\textbf{(M6)}] The regularised map $\psi_{N}$ satisfies
\begin{equation*}
c_{N}^{1}|X|- c_{N}^{2}\leq \left|\frac{\partial\psi_{N}}{\partial Q}(X)\right|\leq C_{N}^{1}|X|+C_{N}^{2}
\end{equation*}
for all $X\in \mathbb{R}^{d\times d}$ and positive constants $c_{N}^{i}$ and $C_{N}^{i}$ which depend on the mollification parameter $N\geq 1$.
\end{itemize}
\end{thm}

\section{A Priori Estimates for the System (S)}\label{secthree} 

As is customary, we procede in a formal manner to obtain useful \emph{a priori} estimates associated with the system (S). The following manipulations hold only for smooth maps $Q$ and $u$, however the resulting energy estimates associated with (S) do indeed hold for suitable approximants $Q^{(\beta)}$ and $u^{(\beta)}$ in Section \ref{secfour} to come. Moreover, for convenience we replace $\psi$ with a convex map $\psi_{\ast}:\symn\rightarrow[-\psi_{0}, \infty)$ satisfying properties {\bf (M1)} to {\bf (M6)} above. Consequently, we shall have {\em no} uniform $L^{\infty}$ information on maps $Q^{(\beta)}$ in space. We recover such information in the limit: see Section \ref{rangeofthemap} below.
\subsection{An Energy Identity for the System (S)}\label{enid}
Firstly, considering the evolution equation for the Q-tensor field, 
\begin{equation*}
\frac{\partial Q}{\partial t} + (u\cdot\nabla)Q-S(Q, \nabla u)= \Gamma\left(L\Delta Q- \theta\left\langle\frac{\partial\psi_{\ast}}{\partial Q}(Q)\right\rangle+\kappa\,Q\right)
\end{equation*}
and taking $L^{2}$ inner products throughout against $-H$, where
\begin{equation*}
H=L\Delta Q - \theta\left\langle\frac{\partial\psi_{\ast}}{\partial Q}(Q)\right\rangle+\kappa\,Q,
\end{equation*}
one finds after integrating by parts that
\begin{align}\label{qlong}
& \frac{d}{dt}\left(\frac{L}{2}\|\nabla Q\|_{2}^{2}+\theta\iid \psi_{\ast}(Q)-\frac{\kappa}{2}\|Q\|_{2}^{2}\right) +\Gamma\left\|L\Delta Q-\theta\left\langle\frac{\partial\psi_{\ast}}{\partial Q}(Q)\right\rangle+\kappa\,Q\right\|_{2}^{2} \notag \\ = & \, \underbrace{L\iid (u\cdot\nabla)Q:\Delta Q\,dx}_{\mathcal{T}_{1}^{+}} +\underbrace{\iid\left(\omega_{0}Q - Q\omega_{0}\right):\left(\theta\left\langle\frac{\partial\psi_{\ast}}{\partial Q}(Q)\right\rangle-\kappa\,Q\right)\,dx}_{\mathcal{I}}\notag \\ & -\underbrace{L\iid(\omega_{0}Q-Q\omega_{0}):\Delta Q\,dx}_{\mathcal{T}_{2}^{+}} - \underbrace{L\xi\iid \left(\omega_{1}\left(Q+\frac{1}{d}I\right)+\left(Q+\frac{1}{d}I\right)\omega_{1}\right):\Delta Q\,dx}_{\mathcal{T}_{3}^{+}} \notag \\ & +\underbrace{2L\xi\iid \mathrm{tr}\left[Q\nabla u\right]\left(Q+\frac{1}{d}I\right):\Delta Q\,dx}_{\mathcal{T}_{4}^{+}}\notag \\ & +\underbrace{\xi\iid \left(\omega_{1}\left(Q+\frac{1}{d}I\right)+\left(Q+\frac{1}{d}I\right)\omega_{1}\right):\left(\theta\left\langle\frac{\partial\psi_{\ast}}{\partial Q}(Q)\right\rangle-\kappa\,Q\right)\,dx}_{\mathcal{T}_{5}^{+}}\notag \\ & -\underbrace{2\xi\iid \left(Q+\frac{1}{d}I\right):\left(\theta\left\langle\frac{\partial\psi_{\ast}}{\partial Q}(Q)\right\rangle-\kappa\,Q\right)\mathrm{tr}\left[Q\nabla u\right]\,dx}_{\mathcal{T}_{6}^{+}}.
\end{align}
Secondly, considering the evolution of the velocity profile
\begin{equation*}
\frac{\partial u}{\partial t}+(u\cdot\nabla)u-\nu\Delta u + \nabla p = \mathrm{div}\left(\tau+\sigma\right)
\end{equation*}
and this time taking inner products in $L^{2}_{\mathrm{div}}$ throughout the equation against $u$, one obtains in a similar way by parts
\begin{align}\label{ulong}
& \quad \frac{1}{2}\frac{d}{dt}\left(\|u\|_{2}^{2}\right)+\nu\|\nabla u\|_{2}^{2} \notag \\ = & \quad \, \underbrace{L\iid \partial_{i} Q_{j\ell}\partial_{k}Q_{\ell j}\partial_{k}u_{i}\,dx}_{\mathcal{T}_{1}^{-}} - \underbrace{L\iid \left(Q\Delta Q-\Delta Q\,Q\right):\nabla u\,dx}_{\mathcal{T}_{2}^{-}} \notag \\ & \quad +\underbrace{L\xi\iid\left(\left(Q+\frac{1}{d}I\right)\Delta Q+ \Delta Q\left(Q+\frac{1}{d}I\right)\right):\nabla u\,dx}_{\mathcal{T}_{3}^{-}} \notag \\ & \quad -\underbrace{2L\xi\iid \mathrm{tr}\left[Q\Delta Q\right]\left(Q-\frac{1}{d}I\right):\nabla u\,dx}_{\mathcal{T}_{4}^{-}} \notag \\ & \quad -\underbrace{\xi\iid \left(\left(Q+\frac{1}{d}I\right)\left(\theta\left\langle\frac{\partial\psi_{\ast}}{\partial Q}(Q)\right\rangle-\kappa\,Q\right)+\left(\theta\left\langle\frac{\partial\psi_{\ast}}{\partial Q}(Q)\right\rangle-\kappa\,Q\right)\left(Q+\frac{1}{d}I\right)\right):\nabla u\,dx}_{\mathcal{T}_{5}^{-}} \notag \\ & \quad \, +\underbrace{2\xi\int_{\id}\tr\left[Q\left(\theta\left\langle\frac{\partial\psi_{\ast}}{\partial Q}(Q)\right\rangle-\kappa\,Q\right)\right]\left(Q+\frac{1}{d}I\right):\nabla u\,dx}_{\mathcal{T}_{6}^{-}}\notag \\ & \quad +\underbrace{L\iid \left(Q\left(\theta\left\langle\frac{\partial\psi_{\ast}}{\partial Q}(Q)\right\rangle-\kappa\,Q\right)-\left(\theta\left\langle\frac{\partial\psi_{\ast}}{\partial Q}(Q)\right\rangle-\kappa\,Q\right)Q\right):\nabla u\,dx.}_{\mathcal{J}}
\end{align}
Let us write $\mathcal{E}:H^{1}\times L^{2}_{\mathrm{div}}\rightarrow\mathbb{R}\cup\{\infty\}$ to denote the functional
\begin{equation*}
\mathcal{E}(Q, u):=\frac{L}{2}\|\nabla Q\|_{2}^{2}+\theta\iid \psi_{\ast}(Q)\,dx-\frac{\kappa}{2}\|Q\|_{2}^{2}+\frac{1}{2}\|u\|_{2}^{2}.
\end{equation*}
By adding the contributions of \eqref{qlong} and \eqref{ulong} together and noting the cancellations $\mathcal{T}^{+}_{j}+\mathcal{T}^{-}_{j}=0$ for $j=1, ..., 6$ together with the null terms $\mathcal{I}=\mathcal{J}=0$, one discovers
\begin{align}\label{lya}
\frac{d\mathcal{E}}{dt}=- \Gamma\left\|L\Delta Q-\theta\left\langle\frac{\partial\psi_{\ast}}{\partial Q}(Q)\right\rangle+\kappa\,Q\right\|_{2}^{2}-\nu\|\nabla u\|_{2}^{2},
\end{align}
which implies that the functional $\mathcal{E}$ is non-increasing along solution trajectories $(Q(t), u(t))$ in $H^{1}\times L^{2}_{\mathrm{div}}$.

At this point, we have paid a price for the cancellations of higher derivative terms (like $\mathcal{T}^{+}_{1}$ and $\mathcal{T}^{-}_{1}$, for example) which would otherwise complicate our subsequent analysis: the functional $\mathcal{E}$ is not of one sign due to the negative contribution of the 2-norm of the tensor field. This would cause a problem in the sequel if we wish to use $\mathcal{E}$ to obtain bounds on $Q$ and $u$ and their distributional gradients in natural function spaces which are uniform in approximation parameters. 

This issue, however, is remedied as follows. Suppose now that the smooth maps $Q=Q^{(\beta)}$ and $u=u^{(\beta)}$ depend on an approximation parameter $\beta\in\{1, 2, 3, ...\}$, but that the initial data $Q_{0}\in H^{1}$ and $u_{0}\in L^{2}_{\mathrm{div}}$ are independent of $\beta$. Equality \eqref{lya} above implies the simple inequality
\begin{equation*}
\frac{d}{dt}\big(\mathcal{E}(Q^{(\beta)}(t), u^{(\beta)}(t))\big)\leq 0,
\end{equation*}
from which we deduce by integration in time,
\begin{align}\label{firststep}
\quad & \frac{L}{2}\|\nabla Q^{(\beta)}(\cdot, t)\|_{2}^{2}+\theta\iid \psi_{\ast}(Q^{(\beta)}(x, t))\,dx-\frac{\kappa}{2}\|Q^{(\beta)}(\cdot, t)\|_{2}^{2}+\frac{1}{2}\|u^{(\beta)}(\cdot, t)\|_{2}^{2} \notag \\ \leq \quad & \frac{L}{2}\|\nabla Q_{0}\|_{2}^{2}+\theta\iid \psi_{\ast}(Q_{0}(x))\,dx-\frac{\kappa}{2}\|Q_{0}\|_{2}^{2}+\|u_{0}\|_{2}^{2}.
\end{align}
Now, noting that $\psi_{\ast}$ satisfies the property \textbf{(M2)}, which implies that
\begin{equation*}
-\theta\left(\Lambda\pi\right)^{d}\psi_{0}\leq \theta\iid\psi_{\ast}(Q^{(\beta)}(x, t))\,dx,
\end{equation*}
and \emph{if} we have in addition that
\begin{equation}\label{importantltwobound}
\sup_{t}\|Q^{(\beta)}(\cdot, t)\|_{2}\leq c_{0},
\end{equation}
where $c_{0}>0$ is some constant independent of $\beta$, we may deduce from \eqref{firststep} using such bounds that 
\begin{equation}
\sup_{t}\left(\frac{L}{2}\|\nabla Q^{(\beta)}(\cdot, t)\|_{2}^{2}+\frac{1}{2}\|u^{(\beta)}(\cdot, t)\|_{2}^{2}\right)\leq C_{0},
\end{equation}
where
\begin{equation*}
C_{0}:=\frac{L}{2}\|\nabla Q_{0}\|_{2}^{2}+\frac{1}{2}\|u_{0}\|_{2}^{2}+\theta\iid\underbrace{\left(\psi_{0}+\psi_{\ast}(Q_{0}(x)\right)}_{\geq 0}\,dx+c_{0}\geq 0.
\end{equation*}
Thus, the existence of the functional $\mathcal{E}$ along with the additional stipulation that $\{Q^{(\beta)}\}_{\beta=1}^{\infty}$ be uniformly bounded in $L^{\infty}_{t}L^{2}_{x}$ allow us to infer that
\begin{equation}\label{qinfo}
\{Q^{(\beta)}\}_{\beta=1}^{\infty} \quad \text{is uniformly bounded in} \quad L^{\infty}_{t}H^{1}_{x}
\end{equation}
and
\begin{equation}\label{uinfo}
\{u^{(\beta)}\}_{\beta=1}^{\infty} \quad \text{is uniformly bounded in} \quad L^{\infty}_{t}L^{2}_{\mathrm{div}}.
\end{equation}
We can use these uniform bounds to glean yet more information from the identity \eqref{lya}. Performing an expansion of the right-hand side, one finds
\begin{align}\label{gleanmore}
\quad & \frac{d}{dt}\big(\mathcal{E}(Q^{(\beta)}(t), u^{(\beta)}(t))\big)+\Gamma L^{2}\|\Delta Q^{(\beta)}\|_{2}^{2}+\Gamma\theta^{2}\left\|\left\langle\frac{\partial\psi_{\ast}}{\partial Q}(Q^{(\beta)})\right\rangle\right\|_{2}^{2}+\Gamma\kappa\|Q^{(\beta)}\|_{2}^{2} \notag\\ = \quad & -\nu\|\nabla u^{(\beta)}\|_{2}^{2}+ 2\Gamma L \kappa \|\nabla Q^{(\beta)}\|_{2}^{2} +2\Gamma L \theta\iid\Delta Q^{(\beta)}:\left\langle\frac{\partial\psi_{\ast}}{\partial Q}(Q^{(\beta)})\right\rangle\,dx\notag \\ & +2\Gamma \theta \kappa \iid \left\langle\frac{\partial\psi_{\ast}}{\partial Q}(Q^{(\beta)})\right\rangle: Q^{(\beta)}\,dx.
\end{align}
Now, by integration by parts, we obtain
\begin{equation*}
\iid \Delta Q^{(\beta)}:\frac{\partial\psi_{\ast}}{\partial Q}(Q^{(\beta)})\,dx=-\sum_{m=1}^{d}\iid\partial_{m} Q^{(\beta)}_{ij}\frac{\partial^{2}\psi_{\ast}(Q^{(\beta)})}{\partial Q_{ji}\partial Q_{\ell k}}\partial_{m} Q^{(\beta)}_{k\ell}\,dx.
\end{equation*}
Since the map $\psi_{\ast}$ is convex, its Hessian is positive definite: see, for instance, \textsc{Rockafellar} (\cite{rockafellar1997convex} Chapter 4, Theorem 4.5). Using the convexity of $\psi$, one may deduce from the above that
\begin{equation*}
2\Gamma L\theta \iid \Delta Q^{(\beta)}:\frac{\partial\psi_{\ast}}{\partial Q}(Q^{(\beta)})\,dx \leq 0.
\end{equation*}
By Young's inequality, one may also obtain the simple estimate
\begin{equation*}
2\Gamma \theta \kappa \iid\left\langle\frac{\partial\psi_{\ast}}{\partial Q}(Q^{(\beta)})\right\rangle: Q^{(\beta)}\,dx \leq \frac{\Gamma\theta^{2}}{2}\left\|\left\langle\frac{\partial\psi_{\ast}}{\partial Q}(Q^{(\beta)})\right\rangle\right\|_{2}^{2} + 2\Gamma\kappa^{2}\|Q^{(\beta)}\|_{2}^{2}.
\end{equation*}
Using these observations, integrating in time across equality \eqref{gleanmore} yields:
\begin{equation}\label{qappme}
\{Q^{(\beta)}\}_{\beta=1}^{\infty} \quad \text{is uniformly bounded in} \quad L^{\infty}_{t} H^{1}_{x} \quad \text{and} \quad L^{2}_{t}H^{2}_{x},
\end{equation}
\begin{equation}\label{uappyou}
\{u^{(\beta)}\}_{\beta=1}^{\infty} \quad \text{is uniformly bounded in} \quad L^{\infty}_{t} L^{2}_{\mathrm{div}}\quad \text{and} \quad L^{2}_{t}H^{1}_{\mathrm{div}},
\end{equation}
and
\begin{equation}\label{psiappher}
\left\{\left\langle\frac{\partial\psi_{\ast}}{\partial Q}(Q^{(\beta)})\right\rangle\right\}_{\beta=1}^{\infty} \quad \text{is uniformly bounded in} \quad L^{2}_{t} L^{2}_{x}. 
\end{equation}
Thus, the moral of this section is that if we can establish identity \eqref{lya} rigorously, along with $L^{\infty}_{t}L^{2}_{x}$-bounds on $\{Q^{(\beta)}\}_{\beta=1}^{\infty}$ and the `convexity' inequality
\begin{equation}\label{convexity}
2\Gamma L\theta \iid \Delta Q^{(\beta)}:\frac{\partial\psi_{\ast}}{\partial Q}(Q^{(\beta)})\,dx \leq 0,
\end{equation}
we are immediately in the familiar setting in which we aim to employ weak compactness arguments to identify candidate solutions for a limiting system.

\subsection{An Estimate for Higher Regularity of Solutions}\label{highreg}
The property of strict physicality becomes of importance when investigating higher regularity of weak solutions of (S). In particular, strict physicality of weak solutions $Q(\cdot, t)$ for $t>0$ implies that $\psi(Q(\cdot, t))$ is as distributionally differentiable as $Q(\cdot, t)$, since the range of the tensor field belongs to a fixed compact subset of $\symn$. 

In what follows, we only consider the case of two spatial dimensions. By noting convenient cancellations in the system (S), knowledge that $\psi(Q(\cdot, t))\in L^{\infty}$ for $t\geq 0$ allows us to prove $Q\in L^{\infty}(0, T; H^{2})\cap L^{2}(0, T; H^{3})$ and $u\in L^{\infty}(0, T; H^{1}_{\mathrm{div}})\cap L^{2}(0, T; H^{2}_{\mathrm{div}})$ for any $T>0$ whenever initial data $(Q_{0}, u_{0})$ lie in the smaller class $H^{2}\times H^{1}_{\mathrm{div}}$ with $Q_{0}$ satsifying $\psi(Q_{0})\in L^{\infty}$.

Our remarks here are once more merely formal, but shall be made rigorous in Section \ref{higherreg}.  Consider the auxiliary functional $\mathcal{F}:H^{2}\times H^{1}_{\mathrm{div}}\rightarrow [0, \infty)$ defined by
\begin{equation*}
\mathcal{F}(Q, u):=\frac{1}{2}\int_{\id}|\nabla u|^{2}\,dx+\frac{L}{2}\int_{\id}|\Delta Q|^{2}\,dx.
\end{equation*}
We assume once again that the maps $Q^{(\beta)}$ and $u^{(\beta)}$ are smooth, with the additional stipulation that $Q^{(\beta)}(\cdot, t)$ be strictly physical on $\id$ for $t\geq 0$. By considering the time derivative of $\mathcal{F}^{(\beta)}(t):=\mathcal{F}(Q^{(\beta)}(\cdot, t), u^{(\beta)}(\cdot, t))$ and noting the identity
\begin{displaymath}
\begin{array}{c}
\displaystyle L\int_{\id}\Delta u^{(\beta)}_{i}\frac{\partial}{\partial x_{j}}\left(\frac{\partial Q^{(\beta)}_{mn}}{\partial x_{i}}\frac{\partial Q^{(\beta)}_{nm}}{\partial x_{j}}\right)\,dx-L\int_{\id}\Delta Q^{(\beta)}: \Delta\left((u^{(\beta)}\cdot\nabla)Q^{(\beta)}\right)\,dx \vspace{2mm}\\ \displaystyle = 2L\int_{\id}\frac{\partial}{\partial x_{k}}\left(\Delta Q^{(\beta)}\right):(u^{(\beta)}\cdot\nabla)\frac{\partial Q^{(\beta)}}{\partial x_{k}}\,dx, 
\end{array}
\end{displaymath}
one can show that
\begin{displaymath}
\begin{array}{c}
\displaystyle \frac{d\mathcal{F}^{(\beta)}}{dt}+\nu\int_{\id}|\Delta u^{(\beta)}|^{2}\,dx +\Gamma L^{2}\int_{\id}|\nabla \Delta Q^{(\beta)}|^{2}\,dx \vspace{2mm} \\ \displaystyle = \int_{\id}(u^{(\beta)}\cdot\nabla)u^{(\beta)}\cdot\Delta u^{(\beta)}\,dx +\Gamma L\kappa\int_{\id}|\nabla Q^{(\beta)}|^{2}\,dx\vspace{2mm} \\ \displaystyle + 2L\int_{\id}\frac{\partial}{\partial x_{k}}\left(\Delta Q^{(\beta)}\right):(u^{(\beta)}\cdot\nabla)\frac{\partial Q^{(\beta)}}{\partial x_{k}}\,dx -\Gamma L\theta\int_{\id}\Delta Q^{(\beta)}:\Delta\left(\left\langle\frac{\partial\psi}{\partial Q}(Q^{(\beta)})\right\rangle\right)\,dx.
\end{array}
\end{displaymath}
Now, under the assumption that $Q^{(\beta)}$ is strictly physical, it can be shown by means of Ladyzhenskaya's inequality that
\begin{equation}
\frac{d\mathcal{F}^{(\beta)}}{dt}(t)\leq C_{0}\mathcal{F}^{(\beta)}(t)^{2}+C_{1},
\end{equation}
where $C_{0}, C_{1}>0$ are independent of the approximation parameter $\beta$. Since $\mathcal{F}^{(\beta)}$ is uniformly bounded in $L^{1}(0, T)$ by \eqref{qappme} and \eqref{uappyou}, an application of Gronwall's inequality yields that the approximants $Q^{(\beta)}$ and $u^{(\beta)}$ are uniformly bounded in $L^{\infty}(0, T; H^{2})\cap L^{2}(0, T; H^{3})$ and $L^{\infty}(0, T; H^{1}_{\mathrm{div}})\cap L^{2}(0, T; H^{2}_{\mathrm{div}})$ for any $T>0$. Such improved uniform bounds give us better compactness and convergence properties when passing to weak solutions in the limit $\beta\rightarrow\infty$.

Now that we have made our initial remarks regarding the formal structure of the system (S), we embark upon a construction of weak solutions.

\section{Existence of Weak Solutions}\label{secfour}

We posit a family of coupled systems which may be considered as approximants to the coupled system (S). Weak compactness arguments identify candidate solutions to the `limit' system (S), which satisfy the coupled system distributionally in the limit by \emph{strong} compactness of the set of approximate solutions. 

Stage I to Stage IV below should be considered as a schematic for the proof of \textsc{Theorem} \ref{theoremone} above. The rest of the details may be sourced from texts such as \textsc{Lions} \cite{lions1} or \textsc{Robinson} \cite{robinson1}.
\begin{remark}
For simplicity, we \emph{only} consider the case $\xi=0$ as the approximation scheme we adopt in this case is compatible with the \emph{strict physicality} argument in Stage III below. 
We pass comment on the reason our comparison principle argument is not suitable for the non co-rotational case in due course.
\end{remark}
\section*{Stage I: Regularised Potential $\psi_{N}$ and Finite-dimensional Velocity Field}\label{stageone}
Let us set out the basic objects which which we employ in this stage. For $\mu\in\mathbb{R}$, consider the Stokes operator eigenfunction problem on $I^{d}=[-\Lambda\pi/2, \Lambda\pi/2]^{d}$ given by
\begin{displaymath}
\left\{
\begin{array}{l}
\displaystyle -\nu\Delta \phi(x)+\nabla p(x)=\mu\phi(x), \vspace{2mm} \\ \nabla\cdot\phi(x)=0, \\
\end{array}
\right.
\end{displaymath}
supplemented with periodic boundary conditions on $I^{d}$. It is known (see, for example \textsc{Temam} \cite{temam1}) that there exists a countably-infinite family of eigenfunctions $\phi_{i, j}$ with corresponding eigenvalues $\mu_{i, j}:=4\pi^{2}\nu |j|^{2}\Lambda^{2}$ of multiplicity $m(i)$ for indices $i, j\geq 1$. We enumerate this countable family of eigenfunctions and eigenvalues by $\{\phi_{m}\}_{m=1}^{\infty}$ and $\{\mu_{m}\}_{m=1}^{\infty}$, respectively. It is also well known that the family $\{\phi_{m}\}_{m=1}^{\infty}$ constitutes an orthonormal basis for $L^{2}_{\mathrm{div}}$ and an orthogonal basis for $H^{1}_{\mathrm{div}}$. Define $H_{M}:=\mathrm{span}\{\phi_{j}\}_{j=1}^{M}$ and the associated orthogonal projection operator $\mathds{P}_{M}:L^{2}_{\mathrm{div}}\rightarrow H_{M}$ by $\mathds{P}_{M}u:=\sum_{k=1}^{M}\left(u, \phi_{k}\right)_{2}\phi_{k}$ for $u\in L^{2}_{\mathrm{div}}$. 

We now state the main result of this stage.
\begin{prop}\label{appsoln}
For any $u_{0}\in L^{2}_{\mathrm{div}}$, $Q_{0}\in H^{1}$ satisfying $\psi(Q_{0})\in L^{1}$ and any fixed positive integer $M\geq 1$, there exist maps $Q^{(M, N)}\in C^{\infty}(\id\times(0, \infty); \symn)$ and $u^{(M, N)}\in C^{1}((0, \infty); H_{M})$ satisfying the system (S$_{M}$) given by
\begin{displaymath}
\left\{
\begin{array}{ll}
 & \displaystyle \partial_{t}Q^{(M, N)}+(u^{(M, N)}\cdot\nabla)Q^{(M, N)}-S(Q^{(M, N)}, \nabla u^{(M, N)}) \vspace{2mm} \\ = & \displaystyle\Gamma\left(L\Delta Q^{(M, N)}-\theta\left\langle\frac{\partial\psi_{N}}{\partial Q}(Q^{(M, N)})\right\rangle+\kappa \,Q^{(M, N)}\right), \vspace{2mm}\\  & \displaystyle \partial_{t}u^{(M, N)}+(u^{(M, N)}\cdot\nabla)u^{(M, N)}+\nabla p^{(M, N)} = \nu\Delta u^{(M, N)}+\mathrm{div}\left(\tau(Q^{(M, N)})+\sigma(Q^{(M, N)})\right), \vspace{2mm}\\& \nabla\cdot u^{(M, N)}=0
\end{array}
\right.
\end{displaymath}
pointwise in $\symn$ and $\mathbb{R}^{d}$ respectively, for all $x\in\id$ and $t\in (0, \infty)$. 
\end{prop}
\begin{proof} The result follows from a rather involved Banach and Schauder fixed-point argument, whose proof we omit. The basic form of the argument can be found in \textsc{Lin and Liu} \cite{linliu1}.
\end{proof}

\section*{Stage II: Passing to the Limit $N\rightarrow\infty$}
We now aim to show that our approximants $Q^{(M, N)}$ and $u^{(M, N)}$ are uniformly bounded in $N$ in natural function spaces associated with the \emph{a priori} estimates of Section \ref{secthree}. By compactness methods (such as those of Banach-Alaoglu or Aubin-Lions), we may then identify candidate maps which are solutions of a limiting system of equations ($\mathrm{S}_{M}$) as the parameter $N\rightarrow\infty$. Throughout this stage, $M\geq 1$ is chosen arbitrarily at the start and remains fixed.
\subsection{Uniform Bounds}
In order to carry out the program sketched in Section \ref{secthree}, we need only demonstrate that $\{Q^{(M, N)}\}_{N=1}^{\infty}$ is uniformly bounded in $L^{\infty}_{t}L^{2}_{x}$, since our approximants are sufficiently regular that both the identity \eqref{lya} holds and
\begin{equation*}
\int_{\id}\Delta Q^{(M, N)}:\frac{\partial\psi_{N}}{\partial Q}(Q^{(M, N)})\,dx = -\sum_{m=1}^{d}\int_{\id}\partial_{m}Q^{(M, N)}_{ij}\frac{\partial^{2}\psi_{N}(Q^{(M, N)})}{\partial Q_{ji}\partial Q_{\ell k}}\partial_{m}Q^{(M, N)}_{k\ell}
\end{equation*}
by an application of classical integration by parts, since $\psi_{N}:\symn\rightarrow [-\psi_{0}, \infty)$ is smooth and convex by property \textbf{(M1)}.

We consider the evolution equation for $Q^{(M, N)}$ and recast it in distributional form, namely
\begin{align}
\quad & \int_{\id}\left(\frac{\partial Q^{(M, N)}}{\partial t}+(u^{(M, N)}\cdot\nabla)Q^{(M, N)}-S(Q^{(M, N)}, \nabla u^{(M, N)})\right):\chi\,dx \notag \\ = \quad & \Gamma\int_{\id}\left(L\Delta Q^{(M, N)}-\theta\left\langle\frac{\partial\psi_{N}}{\partial Q}(Q^{(M, N)})\right\rangle+\kappa\,Q^{(M, N)}\right):\chi\,dx,
\end{align}
for any $\chi\in L^{2}_{t}L^{2}_{x}$. Choosing $\chi$ to be the element
\begin{equation*}
\kappa\,Q^{(M, N)}+\theta\left\langle\frac{\partial\psi_{N}}{\partial Q}(Q^{(M, N)})\right\rangle,
\end{equation*}
one may derive the equality
\begin{align}\label{maamaa}
\quad & \frac{d}{dt}\left(\frac{\kappa}{2}\|Q^{(M, N)}\|_{2}^{2}+\theta\int_{\id}\psi_{N}(Q^{(M, N)})\,dx\right)+\Gamma L\|\nabla Q^{(M, N)}\|_{2}^{2}+\Gamma\theta^{2}\left\|\left\langle\frac{\partial\psi_{N}}{\partial Q}(Q^{(M, N)})\right\rangle\right\|_{2}^{2} \notag \\ = \quad & \Gamma L \theta\int_{\id}\Delta Q^{(M, N)}:\frac{\partial\psi_{N}}{\partial Q}(Q^{(M, N)})\,dx+\Gamma\kappa^{2}\|Q^{(M, N)}\|_{2}^{2}.
\end{align}
Now, since property \textbf{(M2)} of the mollified map $\psi_{N}$ implies that
\begin{equation*}
-\theta(\Lambda\pi)^{d}\psi_{0}\leq \theta\int_{\id}\psi_{N}(Q^{(M, N)})\,dx,
\end{equation*}
one may deduce from \eqref{maamaa} that
\begin{equation*}
\frac{d}{dt}\left[e^{-2\Gamma\kappa t}\left(\frac{\kappa}{2}\|Q^{(M, N)}\|_{2}^{2}+\theta\int_{\id}\psi_{N}(Q^{(M, N)})\,dx\right)\right]\leq -\theta(\Lambda\pi)^{d}\psi_{0}\frac{d}{dt}\left(e^{-2\Gamma\kappa t}\right) 
\end{equation*}
and since by property \textbf{(M3)}
\begin{equation*}
\theta\int_{\id}\psi_{N}(Q_{0})\,dx\leq \theta\int_{\id}\psi(Q_{0})\,dx,
\end{equation*}
we find that $\{Q^{(M, N)}\}_{N=1}^{\infty}$ is uniformly bounded in $L^{\infty}_{t}L^{2}_{x}$. Applying the reasoning of Section \ref{secthree}, we quickly deduce that
\begin{equation}\label{me}
\{Q^{(M, N)}\}_{N=1}^{\infty} \quad \text{is uniformly bounded in} \quad L^{\infty}_{t}H^{1}_{x} \quad \text{and} \quad L^{2}_{t}H^{2}_{x}, \vspace{2mm},
\end{equation}
\begin{equation}\label{you}
\{u^{(M, N)}\}_{N=1}^{\infty} \quad \text{is uniformly bounded in} \quad L^{\infty}_{t}L^{2}_{\mathrm{div}} \quad \text{and} \quad L^{2}_{t}H^{1}_{\mathrm{div}}
\end{equation}
and 
\begin{equation}\label{her}
\left\{\left\langle\frac{\partial\psi_{N}}{\partial Q}(Q^{(M, N)})\right\rangle\right\}_{N=1}^{\infty} \quad \text{is uniformly bounded in} \quad L^{2}_{t}L^{2}_{x}.
\end{equation}
By Banach-Alaoglu compactness, we may extract subsequences of \eqref{me}, \eqref{you} and \eqref{her} which converge weakly to limit points $Q^{(M)}$, $u^{(M)}$ and $Y^{(M)}$ in $L^{2}_{t}H^{1}_{x}$, $L^{2}_{t}L^{2}_{\mathrm{div}}$ and $L^{2}_{t}L^{2}_{x}$, respectively. 

We may in fact extract more information from the boundedness property of $\{u^{(M, N)}\}_{N=1}^{\infty}$ above. Firstly, by orthonormality of the family $\{\phi_{k}\}_{k=1}^{\infty}$ in $L^{2}_{\mathrm{div}}$, the 2-norm of $u^{(M, N)}(\cdot, t)$ is realised as
\begin{equation*}
\|u^{(M, N)}(\cdot, t)\|_{2}^{2}=\sum_{m=1}^{M}|c^{(M, N)}_{m}(t)|^{2},
\end{equation*}
a quantity which we know to be uniformly bounded in $N$. Using the fact that one may construct a norm which is equivalent to the standard norm on $H^{s}_{\mathrm{div}}$ using fractional powers of the Stokes operator, one may use the above observation, together with \eqref{you}, to show that
\begin{equation}\label{uss}
\{u^{(M, N)}\}_{N=1}^{\infty} \quad  \text{is uniformly bounded in} \quad C([0, T]; H^{s}_{\mathrm{div}}) \quad \text{for \emph{any}} \quad s\in \mathbb{R}.
\end{equation}
The uniform bounds of \eqref{me} and \eqref{uss} and allow us to show in turn directly from the $Q^{(M, N)}$-equation that 
\begin{equation*}
\left\{\frac{\partial Q^{(M, N)}}{\partial t}\right\}_{N=1}^{\infty} \quad \text{is uniformly bounded in} \quad L^{2}_{t}L^{2}_{x}
\end{equation*}
and due to the higher spatial derivatives of $Q^{(M, N)}$ in the forcing term of the velocity field equation, the weaker statement that
\begin{equation*}
\left\{\frac{\partial u^{(M, N)}}{\partial t}\right\}_{N=1}^{\infty}\quad \text{is uniformly bounded in} \quad L^{2}_{t}H^{-2}_{\mathrm{div}}.
\end{equation*}
One may verify that the weak limits of $\{\partial_{t}Q^{(M, N)}\}_{N=1}^{\infty}$ in $L^{2}_{t}L^{2}_{x}$ and $\{\partial_{t}u^{(M, N)}\}_{N=1}^{\infty}$ in $L^{2}_{t}H^{-2}_{\mathrm{div}}$ are $\partial_{t}Q^{(M)}$ and $\partial_{t}u^{(M)}$, respectively.
\subsection{Strengthening Convergence}\label{strengthen}
For the purposes of passing to a suitable limiting equation, we must now strengthen our notions of convergence of approximants. For the case of the tensor field, by virtue of the chain of compact embeddings $H^{2}\subset\subset H^{1}\subset\subset L^{2}$, Aubin-Lions compactness allows us to strengthen the weak convergence of $Q^{(M, N)}$ to its limit $Q^{(M)}$ to strong convergence in $L^{2}_{t}H^{1}_{x}$. For the case of the velocity field, since the time derivatives of the co-ordinates of $u^{(M, N)}$ satisfy
\begin{equation*}
c^{(M, N)}(t)= (u_{0}, \phi_{k})_{2}e_{k} + \int_{0}^{t}\frac{dc^{(M, N)}}{dt}(s)\,ds \quad \text{for} \quad 0\leq t\leq T,
\end{equation*}
one may show that $\{c^{(M, N)}\}_{N=1}^{\infty}$ constitutes an equi-continuous family of maps in the space $C([0, T]; \mathbb{R}^{M})$. By Arzel\`{a}-Ascoli compactness, we deduce that $c^{(M, N)}$ converges strongly to some $c^{(M)}$ in $C([0, T]; \mathbb{R}^{M})$, following which one may deduce that $u^{(M)}$ is equal to $c^{(M)}_{k}\phi_{k}$. Furthermore, $u^{(M, N)}$ converges strongly to $u^{(M)}$ in $C([0, T]; H^{s}_{\mathrm{div}})$ for any $s\in\mathbb{R}$. In particular, we have that $u^{(M, N)}\rightarrow u^{(M)}$ in $C([0, T]\times \id)$.
\subsection{The Range of the Map $(x, t)\mapsto Q^{(M)}(x, t)$}\label{rangeofthemap}
In order recover pointwise information on the weak limit $Y^{(M)}\in L^{2}_{t}L^{2}_{x}$, we must show that $Q^{(M)}(x, t)\in \mathsf{D}(\psi)$ almost everywhere. To this end, it is sufficient to show that
\begin{equation*}
\int_{\id}\psi(Q^{(M)}(x, t))\,dx<\infty
\end{equation*}
for almost every $0<t<T$, as $\psi(Q^{(M)}(\cdot, t))$ being integrable implies that $Q^{(M)}(x, t)\in\mathsf{D}(\psi)$ almost everywhere.

We fix $N_{0}\geq 1$, and suppose $N\geq N_{0}$. Since we know that $Q^{(M, N)}\rightarrow Q^{(M)}$ in $L^{2}_{t}H^{1}_{x}$ by section \ref{strengthen}, it follows by the smoothness property \textbf{(M1)} of the mollified map $\psi_{N_{0}}$ that $\psi_{N_{0}}(Q^{(M, N)}(x, t))\rightarrow \psi_{N_{0}}(Q^{(M)}(x, t))$ almost everywhere as $N\rightarrow\infty$. Recalling the uniform bound
\begin{equation*}
\int_{\id}\psi_{N}(Q^{(M, N)}(x, t))\,dx\leq C_{0},
\end{equation*}
for some constant $C_{0}>0$ independent of both $N$ and time, by Fatou's lemma and the monotonicity property \textbf{(M3)} we may deduce that
\begin{equation}
\int_{\id}\psi_{N_{0}}(Q^{(M)}(x, t))\,dx\leq C_{0},
\end{equation}
almost everywhere in time. An application of the monotone convergence theorem finally allows us to deduce that $Q^{(M)}(x, t)\in\mathsf{D}(\psi)$ almost everywhere.
\subsection{The Limiting System as $N\rightarrow \infty$}
We now consider passing to a limiting equation which the weak limits $Q^{(M)}$ and $u^{(M)}$ satisfy. It follows that since $Q^{(M, N)}(x, t)$ tends to an element of $\mathsf{D}(\psi)$ almost everywhere, 
\begin{equation*}
\frac{\partial\psi_{N}}{\partial Q}(Q^{(M, N)}(x, t))\longrightarrow\frac{\partial\psi}{\partial Q}(Q^{(M)}(x, t))\quad \text{almost everywhere on} \quad \id\times (0, T),
\end{equation*} 
by property {\bf (M5)}. By a dominated convergence argument, one may conclude that $Y_{M}=\partial_{Q}\psi(Q^{(M)})$ in $L^{2}_{t}L^{2}_{x}$.


Finally, passing to the limit $N\rightarrow\infty$ we find that the maps $Q^{(M)}\in L^{\infty}_{t}L^{\infty}_{x}\cap L^{\infty}_{t}H^{1}_{x}\cap L^{2}_{t}H^{2}_{x}$ and $u^{(M)}\in L^{\infty}_{t}L^{2}_{\mathrm{div}}\cap L^{2}_{t}H^{1}_{\mathrm{div}}$ satisfy the coupled system
\begin{displaymath}
(\mathrm{S}_{M}) \left\{
\begin{array}{c}
\displaystyle \frac{\partial Q^{(M)}}{\partial t}+(u^{(M)}\cdot\nabla)Q^{(M)}-S_{M} = \Gamma\left(L\Delta Q^{(M)}-\theta\left\langle\frac{\partial\psi}{\partial Q}(Q^{(M)})\right\rangle+\kappa\,Q^{(M)}\right)\vspace{2mm} \\ \displaystyle\frac{\partial u^{(M)}}{\partial t}+(u^{(M)}\cdot\nabla)u^{(M)}-\nu\Delta u^{(M)}+\nabla p = \mathrm{div}\left(\tau(Q^{(M)})+\sigma(Q^{(M)})\right)\vspace{2mm}\\ \nabla\cdot u^{(M)} = 0 
\end{array}
\right.
\end{displaymath}
in the distributional sense. One may also swiftly verify by standard methods that $Q^{(M)}\in C([0, T]; L^{2})$ and $u^{(M)}\in C([0, T]; (L^{2}_{\mathrm{div}})_{w})$.

\section*{Stage III: Strict Physicality of Approximants}
In the previous stage, we were able to demonstrate that $Q^{(M)}(x, t)\in\mathsf{D}(\psi)$ almost everywhere, which was a necessary step before passing to the limit system (S$_{M}$) above. We can in fact show more, namely that there exists $\delta>0$ which is \emph{independent} of $M\geq 1$ such that
\begin{equation}\label{stricteigen}
\delta-\frac{1}{d}\leq \lambda_{j}(Q^{(M)}(x, t))\leq 1-\frac{1}{d}-\delta, \hspace{2mm} \text{for almost every} \hspace{2mm} (x, t) \hspace{2mm} \text{and} \hspace{2mm} j=1, ..., d,
\end{equation} 
i.e. approximate solutions $Q^{(M)}$ are \emph{strictly physical} for almost every $t>0$. We shall obtain strict physicality of approximants $Q^{(M)}$ (and, in turn, of weak solutions in the limit as $M\rightarrow\infty$) by a suitable comparison principle argument. We apply the maximum principle at this stage in the construction, as weak solutions lack sufficient regularity for these methods to be applied.
\subsection{A Parabolic Inequality for $\psi_{N}(Q^{(M, N)})$}
We begin at the level of regularity obtained at the end of Stage I. We once more recast the evolution equation for $Q^{(M, N)}$ in distributional form, namely
\begin{displaymath}
\begin{array}{c}
\displaystyle \int_{0}^{T}\int_{\id}\left(\frac{\partial Q^{(M, N)}}{\partial t}+(u^{(M, N)}\cdot\nabla)Q^{(M, N)}-S(Q^{(M, N)}, \nabla u^{(M, N)})\right):\chi\,dxdt \vspace{2mm} \\ \displaystyle = \Gamma\int_{0}^{T}\int_{\id}\left(L\Delta Q^{(M, N)}-\theta\left\langle\frac{\partial\psi_{N}}{\partial Q}(Q^{(M, N)})\right\rangle+\kappa\,Q^{(M, N)}\right):\chi\,dxdt,
\end{array}
\end{displaymath}
for any $\chi\in L^{2}_{t}L^{2}_{x}$. Making the following choice of test function
\begin{equation*}
\chi=\left\langle\frac{\partial\psi_{N}}{\partial Q}(Q^{(M, N)})\right\rangle\varphi \quad \text{for any}\quad \varphi\in C^{\infty}_{c}(\id\times [0, T], \mathbb{R}_{+}), 
\end{equation*}
one may quickly deduce that the inequality
\begin{equation}
\frac{\partial}{\partial t}\left(\psi_{N}(Q^{(M, N)})\right)+(u^{(M, N)}\cdot\nabla)Q^{(M, N)}-\Gamma L\Delta \psi_{N}(Q^{(M, N)})\leq \frac{\Gamma\kappa^{2}}{2\theta}\tr\left[\left(Q^{(M, N)}\right)^{2}\right]
\end{equation}
holds pointwise in $\mathbb{R}$ for all $(x, t)\in\id\times (0, T)$. Using a suitable comparison function, we now demonstrate that it is possible to control $\psi_{N}(Q^{(M, N)})$ in $L^{\infty}$ \emph{uniformly in} $M$ once we have passed to the limit $N\rightarrow\infty$. This uniform control ultimately allows us to carry information on $\psi(Q^{(M)})$ through to the weak limit $\psi(Q)$.

Consider now the maps $G^{(M, N)}, H^{(M, N)}: \id\times (0, T)\rightarrow \mathbb{R}$, where $G^{(M, N)}$ solves the homogeneous problem
\begin{displaymath}
(\mathrm{P}_{N}^{1})\left\{
\begin{array}{c}
\displaystyle \frac{\partial G^{(M, N)}}{\partial t}+(u^{(M, N)}\cdot\nabla)G^{(M, N)}-\Gamma L\Delta G^{(M, N)}=0, \vspace{2mm} \\ \displaystyle G^{(M, N)}(\cdot, 0)=\psi_{N}(Q_{0})-\frac{1}{(\Lambda\pi)^{d}}\int_{\id}\psi_{N}(Q_{0})\,dx
\end{array}
\right.
\end{displaymath}
with \emph{mean zero} initial data, and $H^{(M, N)}$ solves the inhomogeneous problem
\begin{displaymath}
(\mathrm{P}_{N}^{2})\left\{
\begin{array}{c}
\displaystyle\frac{\partial H^{(M, N)}}{\partial t}+(u^{(M, N)}\cdot\nabla)H^{(M, N)}-\Gamma L\Delta H^{(M, N)}=\frac{\Gamma\kappa^{2}}{2\theta}\tr\left[\left(Q^{(M, N)}\right)^{2}\right] \vspace{2mm} \\\displaystyle H^{(M, N)}(\cdot, 0)=\frac{1}{(\Lambda\pi)^{d}}\int_{\id}\psi_{N}(Q_{0})\,dx
\end{array}
\right.
\end{displaymath}
with constant initial data. Both problems ($\mathrm{P}_{N}^{1}$) and ($\mathrm{P}_{N}^{2}$) are supplemented with periodic boundary conditions on $\id$. Defining the map $K^{(M, N)}$ to be the difference $\psi_{N}(Q^{(M, N)})-G^{(M, N)}-H^{(M, N)}$, one quickly sees that $K^{(M, N)}$ satisfies the parabolic inequality
\begin{equation}\label{gineq}
\frac{\partial K^{(M, N)}}{\partial t}+(u^{(M,N)}\cdot\nabla)K^{(M, N)}-\Gamma L\Delta K^{(M, N)}\leq 0,
\end{equation}
pointwise on $\id\times (0, T)$ with $K^{(M, N)}(\cdot, 0)=0$. By the classical parabolic maximum principle (see, for example, \textsc{Pucci and Serrin} \cite{pucci1}), we find that \eqref{gineq} above implies that
\begin{displaymath}
\begin{array}{c}
\displaystyle K^{(M, N)}(x, t)\leq 0 \quad \text{on} \quad \id\times(0, T) \vspace{2mm} \\ \displaystyle \Longleftrightarrow \quad \psi_{N}(Q^{(M, N)}(x, t))\leq G^{(M, N)}(x, t)+H^{(M, N)}(x, t) \quad \text{on}\quad \id\times (0, T).
\end{array}
\end{displaymath}
Thus, if we can obtain bounds independent of $M$ on $G^{(M, N)}$ and $H^{(M, N)}$ in $L^{\infty}$ in the limit as $N\rightarrow\infty$, we can infer strict physicality of both approximants $Q^{(M)}$ and, in turn, weak solutions $Q$ as $M\rightarrow\infty$. We now perform such an analysis on these two comparison functions.
\subsection{Analysis of the Comparison Function $G^{(M, N)}$}
By a standard construction, one can show that problem ($\mathrm{P}_{N}^{1}$) has a unique solution which is classically smooth for $t>0$. Although we know that the approximants $\psi_{N}(Q_{0})$ lie in $L^{2}$ by property \textbf{(M6)}, we only know $\psi(Q_{0})$ to be in $L^{1}$. Thus, to gain uniform control on $G^{(M, N)}$ in $L^{\infty}$ in the limit as $N\rightarrow\infty$, we require the following $L^{1}\rightarrow L^{\infty}$ estimate from \textsc{Constantin et Al.} \cite{constantin2}.
\begin{lem} Let $v$ be a smooth, spatially-periodic divergence-free velocity field $v$, and let $\gamma>0$. Suppose that $g$ evolves under the associated advection-diffusion equation on the two- or three-dimensional torus, namely
\begin{displaymath}
\left\{
\begin{array}{l}
\displaystyle\frac{\partial g}{\partial t}+(v\cdot \nabla)g-\Gamma L \Delta g=0, \vspace{2mm} \\ g(\cdot, 0)= g_{0}\in L^{1},
\end{array}
\right.
\end{displaymath}
where $g_{0}$ is of zero mean over $\id$. There exists a constant $C=C(\gamma)>0$ which is independent of $v$ such that $g$ satisfies
\begin{equation}\label{heat}
\|g(\cdot, t)\|_{\infty}\leq \frac{C(\gamma)}{t^{\frac{d}{2}+\gamma}}\|g_{0}\|_{1}
\end{equation}
for $t>0$.
\end{lem}
Applying the estimate \eqref{heat} to the solution of ($\mathrm{P}_{N}^{1}$), one can show that 
\begin{equation*}
\left|G^{(M, N)}(x, t)\right|\leq \frac{C}{t^{\frac{d}{2}+\gamma}}\left\|\psi_{N}(Q_{0})-\frac{1}{(\Lambda\pi)^{d}}\int_{\id}\psi(Q_{0}(y))\,dy\right\|_{1},
\end{equation*}
for all $x\in\id$ and $t>0$, which together with property \textbf{(M3)} of the regularised potential yields
\begin{equation}\label{fone}
\left| G^{(M, N)}(x, t)\right|\leq \frac{C}{t^{\frac{d}{2}+\gamma}}\|\psi(Q_{0})\|_{1},
\end{equation}
where the resulting bound is clearly independent of both $M$ and $N$.
\subsection{Analysis of the Comparison Function $H^{(M, N)}$}\label{analcomp}
We now compare the smooth solutions of problem ($\mathrm{P}_{N}^{2}$) with those of the `limiting' problem
\begin{displaymath}
(\mathrm{P}^{2})\left\{
\begin{array}{c}
\displaystyle\frac{\partial H^{(M)}}{\partial t}+(u^{(M)}\cdot\nabla)H^{(M)}-\Gamma L\Delta H^{(M)}=\frac{\Gamma\kappa^{2}}{2\theta}\tr\left[\left(Q^{(M)}\right)^{2}\right], \vspace{2mm} \\ \displaystyle H^{(M)}(\cdot, 0)=\frac{1}{(\Lambda\pi)^{d}}\int_{\id}\psi(Q_{0})\,dx
\end{array}
\right.
\end{displaymath}
Since the initial datum is constant, $Q^{(M)}\in L^{\infty}_{t}H^{1}_{x}\cap L^{2}_{t}H^{2}_{x}$ with $Q^{(M)}(\cdot, t)\in L^{\infty}$, and the vector field $u^{(M)}$ is also smooth in space and time, it follows that $H^{(M)}\in L^{\infty}_{t}H^{3}_{x}\cap L^{2}_{t}H^{4}_{x}$. In particular, we have by the Sobolev embedding theorem that $H^{(M)}(\cdot, t)\in C^{2}_{\mathrm{per}}(\id)$ for $0\leq t<T$. If we denote the difference $H^{(M, N)}-H^{(M)}$ by $\overline{H}^{(M, N)}$, it can be checked that it satisfies the equality
\begin{displaymath}
\begin{array}{c}
\displaystyle \frac{\partial \overline{H}^{(M, N)}}{\partial t}+(u^{(M)}\cdot \nabla)\overline{H}^{(M, N)}-\Gamma L\Delta \overline{H}^{(M, N)} = \frac{\Gamma\kappa^{2}}{2\theta}\tr\left[\left(Q^{(M, N)}\right)^{2}-\left(Q^{(M)}\right)^{2}\right] \vspace{2mm} \\ \displaystyle -\left((u^{(M, N)}-u^{(M)})\cdot \nabla\right)H^{(M, N)}
\end{array}
\end{displaymath}
in $L^{2}$ for each $t>0$. Multiplying throughout by $\overline{H}^{(M, N)}$ and integrating over the spatial domain $\id$, one may use the fact that $\{H^{(M, N)}\}_{N=1}^{\infty}$ is uniformly bounded in $L^{\infty}_{t}L^{2}_{x}$ to derive the inequality
\begin{displaymath}
\begin{array}{c}
\displaystyle \frac{1}{2}\frac{d}{dt}\|\overline{H}^{(M, N)}(\cdot, t)\|_{2}^{2}+\Gamma L \|\nabla\overline{H}^{(M, N)}(\cdot, t)\|_{2}^{2}\leq C\|Q^{(M, N)}(\cdot, t)-Q^{(M)}(\cdot, t)\|_{4}^{2}\vspace{2mm} \\ \displaystyle + C\|\overline{H}^{(M, N)}(\cdot, t)\|_{2}^{2}+ C\left(\max_{[0, T]}\|u^{(M, N)}(\cdot, s)-u^{(M)}(\cdot, s)\|_{\infty}\right)\|\nabla H^{(M)}(\cdot, t)\|_{2},
\end{array}
\end{displaymath} 
which holds for $0<t<T$. Finally, an application of Gronwall's inequality along with the results of section \ref{strengthen} yield the convergence result
\begin{equation*}
\lim_{N\rightarrow \infty}\|H^{(M, N)}(\cdot, t)-H^{(M)}(\cdot, t)\|_{2}=0,
\end{equation*}
from which we deduce that
\begin{equation}\label{ftwo}
H^{(M, N)}(x, t)\rightarrow H^{(M)}(x, t) \quad \text{almost everywhere on} \quad \id\times (0, T) \quad \text{as} \quad N\rightarrow \infty.
\end{equation}
At this point, we can say the following. Since $\psi_{M}(Q^{(M, N)}(x, t))\leq G^{(M, N)}(x, t)+H^{(M, N)}(x, t)$, using property \textbf{(M4)} of the regularised potential and the fact that $Q^{(M, N)}(x, t)\rightarrow Q^{(M)}(x, t)\in\mathsf{D}(\psi)$ almost everywhere as $N\rightarrow\infty$, we may deduce that
\begin{displaymath}
\begin{array}{c}
\displaystyle \psi_{N}(Q^{(M, N)}(x, t))\leq \frac{C}{t^{\frac{d}{2}+\gamma}}\|\psi(Q_{0})\|_{1}+H^{(M, N)}(x, t) \vspace{2mm} \\ \displaystyle \Longrightarrow \psi(Q^{(M)}(x, t)) = \lim_{N\rightarrow\infty}\psi_{N}(Q^{(M, N)}(x, t))\leq \frac{C}{t^{\frac{d}{2}+\gamma}}\|\psi(Q_{0})\|_{1}+H^{(M)}(x, t)
\end{array}
\end{displaymath}
for almost every $(x, t)\in \id\times (0, T)$ by results \eqref{fone} and \eqref{ftwo} above. Therefore, if we can find a bound on $H^{(M)}$ in $L^{\infty}$ which is independent of $M$, we have demonstrated our earlier claim that $\psi(Q^{(M)}(\cdot, t))$ may be controlled in $L^{\infty}$ uniformly in $M$.
\subsection{Uniform Bounds in $L^{\infty}$ on $H^{(M)}$}
As solutions of problem ($\mathrm{P}^{2}$) are sufficiently regular, multiplying throughout the equation by $H^{(M)}((H^{(M)})^{2})^{p/2-1}$ for $p>2$ and integrating over $\id$, we find
\begin{displaymath}
\begin{array}{c}
\displaystyle \int_{\id}\left(\frac{1}{p}\frac{\partial}{\partial t}\left(H^{(M)}\right)^{p}\right)+\frac{1}{p}u^{(M)}_{k}\frac{\partial}{\partial x_{k}}\left(\left(H^{M})^{p}\right)-\Gamma L \Delta H^{(M)}H^{(M)}(\left(H^{(M)}\right)^{2})^{\frac{p}{2}-1}\right)\,dx\vspace{2mm} \\ \displaystyle \leq \frac{\Gamma\kappa^{2}}{2\theta}\int_{\id}\tr\left[\left(Q^{(M)}\right)^{2}\right] H^{(M)}(\left(H^{(M)}\right)^{2})^{\frac{p}{2}-1}\,dx,
\end{array}
\end{displaymath}
whence
\begin{displaymath}
\begin{array}{c}
\displaystyle \frac{1}{p}\frac{d}{dt}\|H^{(M)}(\cdot, t)\|_{p}^{p} \leq \left(\frac{\Gamma\kappa^{2}}{2\theta}\right)^{p}\frac{1}{p}\int_{\id}\left(\tr\left[\left(Q^{(M)}\right)^{2}\right]\right)^{p}\,dx+\left(1-\frac{1}{p}\right)\|H^{(M)}(\cdot, t)\|_{p}^{p}\vspace{3mm} \\ \displaystyle \Longrightarrow \frac{1}{p}\frac{d}{dt}\|H^{(M)}(\cdot, t)\|_{p}^{p}\leq \left(\frac{\Gamma\kappa^{2}}{2\theta}\right)^{p}\frac{(\Lambda\pi)^{d}}{p}+\left(1-\frac{1}{p}\right)\|H^{(M)}(\cdot, t)\|_{p}^{p} \vspace{3mm} \\ \displaystyle \Longrightarrow \|H^{(M)}(\cdot, t)\|_{p}\leq 2^{\frac{1}{p}}\left(\left|\int_{\id}\psi(Q_{0})\,dx\right|e^{(1-1/p)T}+\left(\frac{\Gamma\kappa^{2}}{2\theta}\right)\frac{e^{(1-1/p)T}}{(p-1)^{\frac{1}{p}}}\right).
\end{array}
\end{displaymath}
Finally, using the fact that $H^{(M)}$ lies in $L^{p}$ for \emph{every} $p\in [1, \infty)\cup\{\infty\}$, by taking the limit $p\rightarrow\infty$ in the above inequality, we deduce that solutions of ($\mathrm{P}^{2}$) satisfy
\begin{equation}
\|H^{(M)}(\cdot, t)\|_{\infty}\leq \left|\int_{\id}\psi(Q_{0})\,dx\right|e^{T}+\left(\frac{\Gamma\kappa^{2}}{2\theta}\right)e^{T},
\end{equation}
and so the sequence $\{H^{(M)}\}_{M=1}^{\infty}$ is indeed uniformly bounded in $L^{\infty}_{t}L^{\infty}_{x}$ as claimed.

Piecing together the remarks of this stage, we conclude that
\begin{equation}\label{yesest}
\psi(Q^{(M)}(x, t)) \leq \frac{C}{t^{\frac{d}{2}+\gamma}}\|\psi(Q_{0})\|_{1}+\left|\int_{\id}\psi(Q_{0})\,dx\right|e^{T}+\left(\frac{\Gamma\kappa^{2}}{2\theta}\right)e^{T},
\end{equation}
whence $Q^{(M)}(\cdot, t)$ is strictly physical for $t>0$. Such a property will automatically be inherited by weak solutions $Q$ if we can show $Q^{(M)}(x, t)\rightarrow Q(x, t)$ almost everywhere as $M\rightarrow \infty$. 
\begin{rem}
Let us comment briefly on the non co-rotational case $\xi\in\mathbb{R}\setminus\{0\}$. If one follows through the scheme of Stage III, one finds that the relevant comparison function $H$ in Section \ref{analcomp} should satisfy
\begin{displaymath}
\left\{
\begin{array}{c}
\displaystyle \frac{\partial H}{\partial t}+(u^{(M)}\cdot\nabla)H-\Gamma L \Delta H = c_{0}(\xi)\tr\left[(\nabla u^{(M)})^{2}\right] + \frac{\Gamma\kappa^{2}}{2\theta}\tr\left[\left(Q^{(M)}\right)^{2}\right] \vspace{2mm} \\ \displaystyle H(\cdot, 0)= \frac{1}{(\Lambda \pi)^{d}}\int_{\id}\psi(Q_{0})\,dx,
\end{array}
\right.
\end{displaymath} 
where $c_{0}(0)=0$. Our construction of weak solutions does not provide us with $W^{1, 4}_{\mathrm{div}}$-bounds on approximants $u^{(M)}$ \emph{uniform} in $M$ which we need to carry out the program outlined above. It is for this reason we restrict our study of strict physicality of weak solutions to the case $\xi=0$. 
\end{rem}

\section*{Stage IV: Passing to the Limit $M\rightarrow\infty$}
At this point, compared with Stage II we have much less to do in order to identify candidate maps for weak solutions. Since we have shown in the previous stages that $Q^{(M)}(x, t)\in\mathsf{D}(\psi)$ almost everywhere, it follows that $Q^{(M)}\in L^{\infty}_{t}L^{\infty}_{x}$ and so is automatically in $L^{\infty}_{t}L^{2}_{x}$. Furthermore, our maps $Q^{(M)}$ and $u^{(M)}$ are sufficiently regular that the identity \eqref{lya} holds. It is also now straightforward to verify that
\begin{equation}\label{sign}
\int_{\id}\Delta Q^{(M)}:\frac{\partial\psi}{\partial Q}(Q^{(M)})\,dx
\end{equation}
is of positive sign. Since we know the image of the maps $Q^{(M)}$ belong to a compact subset of $\mathsf{D}(\psi)$ which is independent of $M$, we need not worry about distributional differentiability of $\partial_{Q}\psi (Q^{(M)})$. Once again, the reasoning of Section \ref{secthree} allows us to infer that
\begin{equation}\label{us}
\{Q^{(M)}\}_{M=1}^{\infty}\quad \text{is uniformly bounded in} \quad L^{\infty}_{t}H^{1}_{x}\cap L^{2}_{t}H^{2}_{x},
\end{equation}
\begin{equation}\label{we}
\{u^{(M)}\}_{M=1}^{\infty} \quad \text{is uniformly bounded in} \quad L^{\infty}_{t}L^{2}_{\mathrm{div}}\cap L^{2}_{t}H^{1}_{\mathrm{div}}
\end{equation}
and also
\begin{equation}\label{them}
\left\{\frac{\partial\psi}{\partial Q}(Q^{(M)})\right\}_{M=1}^{\infty} \quad \text{is uniformly bounded in}\quad L^{2}_{t}L^{2}_{x}.
\end{equation}
We shall denote the weak limit points of the sequences \eqref{we}, \eqref{you} and \eqref{them} by $Q$, $u$ and $Y$, respectively. 

Using the fact that the approximate tensor field equation holds in the strong sense in $L^{2}_{t}L^{2}_{x}$, we may verify that
\begin{equation}\label{meme}
\left\{\frac{\partial Q^{(M)}}{\partial t}\right\}_{M=1}^{\infty} \quad \text{is uniformly bounded in} \quad L^{2}_{t}H^{p(d)}_{x},
\end{equation}
where $p(2)=0$ and $p(3)=-1$. Similarly, considering the distributional form of the approximate velocity field equation, one may in turn check that
\begin{equation}\label{youyou}
\left\{\frac{\partial u^{(M)}}{\partial t}\right\}_{M=1}^{\infty} \quad \text{is uniformly bounded in} \quad L^{4/3}_{t}H^{-1}_{\mathrm{div}}.
\end{equation}
To strengthen the convergence of the approximants $Q^{(M)}$ and $u^{(M)}$ to their weak limits, we need only apply once again the Aubin-Lions compactness lemma.

Finally, using the facts that $Q^{(M)}\rightarrow Q$ in $L^{2}_{t}H^{1}_{x}$, $u^{(M)}\rightarrow u$ in $L^{2}_{t}L^{2}_{\mathrm{div}}$ and that both $Q^{(M)}(x, t)$ and $Q(x, t)$ belong to a compact subset of $\mathsf{D}(\psi)$ almost everywhere, passing to the limit in system ($\mathrm{S}_{M}$) we can show that the system
\begin{displaymath}
\left\{
\begin{array}{c}
\displaystyle \partial_{t}Q+(u\cdot\nabla)Q-S(Q, \nabla u)=\Gamma\left(L\Delta Q-\theta\left\langle\frac{\partial \psi}{\partial Q}(Q)\right\rangle +\kappa\,Q\right), \vspace{2mm} \\ \displaystyle \partial_{t}u+(u\cdot \nabla)u+\nabla p=\nu \Delta u +\mathrm{div}\left(\tau +\sigma\right), \vspace{2mm} \\ \nabla \cdot u = 0
\end{array}
\right.
\end{displaymath}
is satisfied distributionally by $Q\in L^{\infty}_{t}H^{1}_{x}\cap L^{2}_{t}H^{2}_{x}$ and $u\in L^{\infty}_{t}L^{2}_{\mathrm{div}}\cap L^{2}_{t}H^{1}_{\mathrm{div}}$. 

This completes the construction of weak solutions and, in turn, closes the proof of \textsc{Theorem} \ref{theoremone}.

\section{Higher Regularity of Weak Solutions in Dimension 2}\label{higherreg}

As intimated in Section \ref{highreg}, if one places higher regularity conditions on the initial data for the system (S), strict physicality enables one to prove that weak solutions of \textsc{Theorem} \ref{theoremone} are in turn more regular. Furthermore, higher regularity of solutions allows one to infer that the limiting $Q$ equation holds in the strong sense as an equality in $L^{2}_{t}L^{2}_{x}$.

We begin our enquiries at the end of Stage III, under the additional assumption that $u_{0}\in H^{1}_{\mathrm{div}}$ and $Q_{0}\in H^{2}$ with $\psi(Q_{0})\in L^{\infty}$. In particular, the reasoning of this stage (in particular, estimate \eqref{yesest}) implies that $\psi(Q^{(M)}(\cdot, t))\in L^{\infty}$ for $t\geq 0$. Moreover, by comparing approximants $Q^{(M)}$ with strong solutions $R^{(M)}$ of the problem
\begin{displaymath}
\left\{
\begin{array}{ll}
& \displaystyle \frac{\partial R^{(M)}}{\partial t} - \Gamma L \Delta R^{(M)}=-(u^{(M)}\cdot\nabla)Q^{(M)}+S(Q^{(M)}, \nabla u^{(M)})\vspace{2mm}\\
- & \displaystyle \Gamma\left(\theta\left\langle\frac{\partial\psi}{\partial Q}(Q^{(M)})\right\rangle-\kappa\,Q^{(M)}\right), \vspace{2mm}\\
& R^{(M)}(\cdot, 0)=Q_{0},
\end{array}
\right.
\end{displaymath}
one may deduce by uniqueness that $Q^{(M)}\in L^{\infty}_{t}H^{2}_{x}\cap L^{2}_{t}H^{3}_{x}$ for each $M\geq 1$.

With this in mind, consider the `higher-order' energy
\begin{equation*}
\mathcal{F}^{(M)}(t):=\mathcal{F}(Q^{(M)}(\cdot, t), u^{(M)}(\cdot, t))=\frac{1}{2}\int_{\id}|\nabla u^{(M)}(x, t)|^{2}\,dx+\frac{L}{2}\int_{\id}|\Delta Q^{(M)}(x, t)|^{2}\,dx.
\end{equation*}
Using the identity
\begin{displaymath}
\begin{array}{c}
\displaystyle L\int_{\id}\Delta u^{(M)}_{i}\frac{\partial}{\partial x_{j}}\left(\frac{\partial Q^{(M)}_{mn}}{\partial x_{i}}\frac{\partial Q^{(M)}_{nm}}{\partial x_{j}}\right)\,dx-L\int_{\id}\Delta Q^{(M)}: \Delta\left((u^{(M)}\cdot\nabla)Q^{(M)}\right)\,dx \vspace{2mm}\\ \displaystyle = 2L\int_{\id}\frac{\partial}{\partial x_{k}}\left(\Delta Q^{(M)}\right):(u^{(M)}\cdot\nabla)\frac{\partial Q^{(M)}}{\partial x_{k}}\,dx, 
\end{array}
\end{displaymath}
one can show that, at this level of regularity, the following energy identity holds:
\begin{displaymath}
\begin{array}{c}
\displaystyle \frac{d\mathcal{F}^{(M)}}{dt}+\nu \int_{\id}|\Delta u^{(M)}|^{2}\,dx+\Gamma L^{2}\int_{\id}|\nabla \Delta Q^{(M)}|^{2}\, dx \vspace{2mm} \\ \displaystyle = \underbrace{\int_{\id}(u^{(M)}\cdot\nabla)u^{(M)}\cdot\Delta u^{(M)}\,dx}_{\mathcal{K}_{1}:=} + \Gamma L\kappa\int_{\id}|\nabla Q^{(M)}|^{2}\,dx \vspace{2mm} \\ \displaystyle +\underbrace{2L\int_{\id}(u^{(M)}\cdot\nabla)\frac{\partial}{\partial x_{k}}Q^{(M)}:\frac{\partial}{\partial x_{k}}\Delta Q^{(M)}\,dx}_{\mathcal{K}_{2}:=}-\Gamma L\theta\int_{\id}\Delta Q^{(M)}:\Delta \left(\left\langle \frac{\partial\psi}{\partial Q}(Q^{(M)})\right\rangle\right)\,dx.
\end{array}
\end{displaymath}
By means of Ladyzhenskaya's inequality and the uniform bounds \eqref{us} and \eqref{we}, one may derive the estimate
\begin{equation*}
\mathcal{K}_{1}\leq\frac{\nu}{2}\|\Delta u^{(M)}\|_{2}^{2}+C\left(\|\nabla u^{(M)}\|_{2}^{4}+1\right)
\end{equation*}
and also
\begin{equation*}
\mathcal{K}_{2}\leq \frac{\Gamma L^{2}}{2}\|\nabla \Delta Q^{(M)}\|_{2}^{2}+ C\left(\|\nabla u^{(M)}\|_{2}^{4}+\|\Delta Q^{(M)}\|_{2}^{4}\right),
\end{equation*}
for constants $C>0$ which are independent of $M$. Furthermore, since strict physicality implies the existence of a compact subset $K\subset\symn$ such that $Q^{(M)}(x, t)\in K$ for {\em all} $M\geq 1$ and a.e. $(x, t)\in I^{2}\times [0, T]$, one has by property \textbf{(P1)} of $\psi$ that
\begin{equation*}
\left|\frac{\partial^{k}\psi}{\partial Q_{i_{1}j_{1}}...\partial Q_{i_{k}j_{k}}}(Q^{(M)}(x, t))\right|\leq C_{k} \quad \text{for} \hspace{2mm} k\in\{1, 2, 3\},
\end{equation*}
for some constant independent of $M$. Such information leads one to deduce that
\begin{equation}\label{gronny}
\frac{d\mathcal{F}^{(M)}}{dt}(t)\leq C_{0}\left(\mathcal{F}^{(M)}(t)\right)^{2}+C_{1}
\end{equation}
for $M$-independent constants $C_{i}>0$. Multiplying throughout \eqref{gronny} by the integration factor
\begin{equation*}
\exp\left(-C_{0}\int_{0}^{t}\mathcal{F}^{(M)}(s)\,ds\right),
\end{equation*} 
acknowledging left-continuity of $\mathcal{F}^{(M)}$ at 0 and using the fact that $\mathcal{F}^{(M)}$ is uniformly bounded in $L^{1}(0, T)$ for any $T>0$, an application of Gronwall's inequality allows one to deduce that
\begin{equation}
\left\{Q^{(M)}\right\}_{M=1}^{\infty} \quad \text{is uniformly bounded in} \quad L^{\infty}_{t}H^{2}_{x}\cap L^{2}_{t}H^{3}_{x},
\end{equation}
and 
\begin{equation}
\left\{u^{(M)}\right\}_{M=1}^{\infty} \quad \text{is uniformly bounded in} \quad L^{\infty}_{t}H^{1}_{\mathrm{div}}\cap L^{2}_{t}H^{2}_{\mathrm{div}}.
\end{equation} 

Such uniform bounds give us more with which to work as we pass to limiting weak solutions of system (S). Aubin-Lions compactness guarantees us the existence of (relabeled) subsequences of $\{Q^{(M)}\}_{M=1}^{\infty}$ and $\{u^{(M)}\}_{M=1}^{\infty}$ which are strongly convergent in $L^{2}_{t}H^{2}_{x}$ and $L^{2}_{t}H^{1}_{\mathrm{div}}$, respectively. In particular, one may show that $\{\partial_{t}Q^{(M)}\}_{M=1}^{\infty}$ is uniformly bounded in $L^{2}_{t}L^{2}_{x}$, and further that the limiting $Q$ equation holds in the strong sense as an equality in $L^{2}_{t}L^{2}_{x}$.



\subsubsection*{Acknowledgements}
I extend my thanks to both the University of Sussex in Brighton, at which I was a visiting Research Fellow, and also the Hausdorff Institute in Bonn, at which I was a visitor, for their excellent hospitality whilst this work was being completed. Finally, I would like to thank Arghir Zarnescu and John Ball for interesting discussions related to the material in this paper.


\bibliography{physic}
\end{document}